\documentclass[french, 11pt, a4paper, oneside,bibliography=totocnumbered]{scrartcl}
\usepackage[T1]{fontenc} 
\usepackage[latin1]{inputenc} 
\usepackage{mathptmx} 
\usepackage{lmodern} 
\usepackage{amsthm} 
\usepackage{amsmath} 
\usepackage{amsfonts} 
\usepackage{amssymb} 
\usepackage{mathrsfs}
\usepackage{bbm}
\usepackage[left=2.7cm, right=2.7cm, top=2.7cm, bottom=2.7cm]{geometry} 
\usepackage[ngerman,german,english]{babel} 
\usepackage{paralist} 
\usepackage[colorlinks=true,citecolor=blue]{hyperref}
\hypersetup{linktocpage}
\usepackage{etoolbox}
\usepackage[nottoc,notlot,notlof]{tocbibind}
\usepackage{graphicx}

\setkomafont{sectionentry}{\normalsize}
\RedeclareSectionCommand[tocbeforeskip=1pt]{section}

\setkomafont{section}{\normalfont\Large\textbf}
\setkomafont{subsection}{\normalfont\large\textbf}
\setkomafont{paragraph}{\normalfont\textbf}

\renewenvironment{abstract}
{\begin{center}
		\textbf{Abstract}
	\end{center}
	\list{}{ 
		\setlength{\leftmargin}{0.05\textwidth}
		\setlength{\rightmargin}{\leftmargin}
	}
	\item\relax} 
{\endlist}

\newenvironment{keywords}
{\begin{trivlist}\item[]{\bfseries Keywords.}}
	{\end{trivlist}}

\newenvironment{subclass}
{\begin{trivlist}\item[]{\bfseries Mathematical Subject Classification:}}
	{\end{trivlist}}

\newenvironment{acknowledgements}
{\begin{trivlist}\item[]{\bfseries Acknowledgements:}}
	{\end{trivlist}}

\numberwithin{equation}{section}
\theoremstyle{plain} 
\newtheorem{thm}{Theorem}[section] 
\newtheorem{cor}[thm]{Corollary}
\newtheorem{lem}[thm]{Lemma}
\newtheorem{pro}[thm]{Proposition}

\theoremstyle{definition}

\newtheorem{rem}[thm]{Remark}

\newcommand{\C}{\mathbb{C}}
\newcommand{\N}{\mathbb{N}}

\newcommand{\R}{\mathbb{R}}
\newcommand{\Z}{\mathbb{Z}}
\newcommand{\F}{\mathbb{F}}
\newcommand{\T}{\mathbb{T}}
\newcommand{\D}{\mathbb{D}}

\newcommand\restr[2]{{
		\left.\kern-\nulldelimiterspace 
		#1 
		\vphantom{\big|} 
		\right|_{#2} 
}}

\newcommand{\norm}[2][]{\left\|#2\right\|_{#1}}

\newcommand{\semnorm}[2][]{\left[#2\right]_{#1}}

\newcommand{\Lpnorm}[2][]{\ifthenelse{\equal{#1}{}}{\norm{#2}_{L^p}}{\norm{#2}_{L^p(#1)}}}
\newcommand{\Hknorm}[2][]{\ifthenelse{\equal{#1}{}}{\norm{#2}_{H^k}}{\norm{#2}_{H^k(#1)}}}

\newcommand{\sgn}{\mathrm{sgn}}

\renewcommand{\Re}{\text{\normalfont Re}}

\newcommand{\divergence}{\text{\normalfont div}}

\newcommand{\QV}[2][]{\ifthenelse{\equal{#1}{}}{\langle #1 \rangle}{\langle #1,#2 \rangle}}

\newcommand{\ind}{\mathbbm{1}}

\newcommand{\Lin}{\mathscr{L}}

\makeatletter
\newcounter{author}
\renewcommand*\author[1]{%
	\stepcounter{author}%
	\ifnum\c@author=1
	\gdef\@author{#1}%
	\else
	\xdef\@author{\unexpanded\expandafter{\@author\and#1}}%
	\fi
	\csgdef{author@\the\c@author}{#1}}
\newcommand*\email[1]{%
	\csgdef{email@\the\c@author}{#1}}
\newcommand*\address[1]{%
	\csgdef{address@\the\c@author}{#1}}
\AtEndDocument{%
	\xdef\author@count{\the\c@author}%
	\c@author=1
	\print@authors}
\newcommand*\print@authors{%
	\ifnum\c@author>\author@count
	\else
	\print@author{\the\c@author}%
	\advance\c@author by 1
	\expandafter\print@authors
	\fi}
\newcommand*\print@author[1]{%
	\par\medskip
	\begin{tabular}{@{}l@{}}%
		\textsc{\csuse{author@#1}}\\
		\csuse{address@#1}\\
		\textit{E-mail address}: \csuse{email@#1}
\end{tabular}}
\makeatother

\title{\textrm{\textbf{\Large Rotating solutions to the incompressible Euler-Poisson equation with external particle}}}

\author{\large Diego Alonso-Or\'an}
\address{Universidad de La Laguna, Departamento de An\'alisis Matem\'atico \\ C/Astrof\'isico Francisco S\'anchez s/n \\ 38271 La Laguna \\ SPAIN}
\email{dalonsoo@ull.edu.es}

\author{\large Bernhard Kepka}
\address{University of Bonn, Institute for Applied Mathematics \\ Endenicher Allee 60 \\ D-53115 Bonn \\ GERMANY}
\email{kepka@iam.uni-bonn.de}

\author{\large Juan J. L. Vel\'{a}zquez}
\address{University of Bonn, Institute for Applied Mathematics \\ Endenicher Allee 60 \\ D-53115 Bonn \\ GERMANY}
\email{velazquez@iam.uni-bonn.de}

\date{\normalsize \today}

\begin{document}
	\maketitle
	
	\begin{abstract}
		We consider a two-dimensional, incompressible fluid body, together with self-induced interactions. The body is perturbed by an external particle with small mass. The whole configuration rotates uniformly around the common center of mass. We construct solutions, which are stationary in a rotating coordinate system, using perturbative methods. In addition, we consider a large class of internal motions of the fluid. The angular velocity is related to the position of the external particle and is chosen to satisfy a non-resonance condition.
		
		\begin{subclass}
			35Q35, 35Q31, 76B07.
		\end{subclass}
		
		\begin{acknowledgements}
			D. Alonso-Or\'{a}n is supported by the Spanish MINECO \linebreak through Juan de la Cierva fellowship FJC2020-046032-I. B. Kepka and J. J. L. Vel\'azquez gratefully acknowledge the support by the Deutsche Forschungsgemeinschaft (DFG) through the collaborative research centre The mathematics of emerging effects (CRC 1060, Project-ID 211504053). B. Kepka is funded by the Bonn International Graduate School of Mathematics at
			the Hausdorff Center for Mathematics (EXC 2047/1, Project-ID 390685813). J. J. L. Vel\'azquez is funded by the DFG under Germany's Excellence Strategy-EXC2047/1-390685813.
		\end{acknowledgements}
	\end{abstract}
	\tableofcontents
	
	\section{Introduction and previous results}
	The shape of fluid objects due to the combination of rotational and self-gravitating forces is a classical research field which has been extensively considered for different fluid models. In particular, a detailed description of the historical evolution of the field can be found \cite{Chandrasekhar1969} for the (three-dimensional) incompressible Euler equations. Further results were established by Lichtenstein \cite{Lichtenstein1918UntersuchungenI}. For the case of compressible fluids we refer to the works \cite{AuchmutyBeals1971ModelsRotatingStars,ChanilloLi1994OnDiametersOfUniformly,Heilig1994Lichtenstein,JangMakino2017SlowlyRotatingAxissymmetricSol,JangMakino2019RotatingAxisymmetricSol,JangSeok2022RotatingBinaryStarsGalaxies,Li1991OnUniformlyRotatingStars,Lichtenstein1933UntersuchungenIII,LuoSmoller2004RotatingFluids,LuoSmoller2009ExistenceNonlinearStabilityRotatingStar,StraussWu2017SteadyStatesRotatingStarsGalaxies,StraussWu2019RapidlyRotatingStars} and references therein. A kinetic model, namely the Vlasov-Poisson equation, has been studied as well, see e.g. \cite{Dolbeault2008LoclaizedMinimizersFlatRotating,JangSeok2022RotatingBinaryStarsGalaxies}. In fact, there is a relation between steady states of the Vlasov-Poisson equation and the compressible Euler equation, see \cite{Rein2007} and references therein for an overview of the variational methods used in these problems.
	
	In this paper, we consider a two-dimensional, self-interacting, incompressible fluid body modeled by the Euler equations. Furthermore, we study the problem of deformations of the geometry when it is perturbed by some external particle. The fluid body and the external particle are assumed to rotate around their center of mass. This problem (adding a small particle) can be seen as a test of stability of the rotating solutions and also as a simple model of tides. Furthermore, differently from the results reviewed in \cite{Chandrasekhar1969} (excluding the figures studied by Riemann), we construct solutions of the Euler-Poisson equation for which the fluid velocity is in general different from zero in any coordinate system.
	
	In this work, we study a family of interaction potentials including the classical (Newtonian) gravitational forces. The latter can be interpreted as an extremely simplified model for galaxies. However, this does not correspond to a three-dimensional problem restricted to planar geometries. The reason being that the pressure would necessary act only in the plane which contains the fluid body as well as the external particle. Nevertheless, such a model can be considered in the case of the Vlasov-Poisson equation, assuming that the velocities of the particles are contained only in the same plane as the fluid. In this situation the tensor describing the pressure is anisotropic and it yields zero forces in the direction perpendicular to the plane but not in the horizontal direction, cf. \cite{Rein1999FlatSteadyStates}. 

	Since we consider a two-dimensional fluid body we can apply two tools that cannot be employed in three dimensional problems. Specifically, we use conformal mappings as well as the Grad-Shafranov method \cite{Grad1967Plasma,Safra}.

    Beside the problem treated here, a variety of different free-boundary problems arising in fluid mechanics have been studied in the last decades. For instance, the problem of jets and cavities with or without gravity has been studied in \cite{AltCaffarellFriedman1982AsymJetFlows,AltCaffarellFriedman1982JetGravity,AltCaffarellFriedman1982AxiallySymmJetFlows} and the theory of gravity water waves has been developed in several works, cf. \cite{IonescuPusateri, Sijue1, Sijue2}. An important difference between the previous free-boundary problems and the one studied in this paper is that the interacting force (e.g. gravity) is due to the fluid itself.  Another type of problems that have some similarities with the one considered in this article are those related to the theory of rotating vortex patches. The first rigorous result was shown by Burbea \cite{burbea1982motions} where he constructed rotating vortex patches close to the disk by means of the classical Crandall-Rabinowitch bifurcation approach. A more thorough study of rotating vortex patches can be found in \cite{hassainia2020global,hmidi2013boundary} and the references therein.
    
	\subsection{Setting of the problem}
	We are concerned with a flat incompressible fluid body with density $\rho=\ind_E $. Here, $ \ind_E $ denotes the indicator function of the set $ E $. The shape of the body $ E(t)\subset \R^2 $ has a smooth boundary and is close to a disk, see below for the precise meaning of this. We also include a particle $ X=X(t) \in \R^2 $ with small mass $ m $. However, we consider only situations in which the particle and the fluid body are at a positive distance. The velocity field $ v $ of the fluid body then satisfies the following free-boundary problem for the Euler-Poisson system
	\begin{align}\label{eq:IncompEulerPoissonPrev}
		\begin{cases}
			\partial_t v + (v\cdot \nabla) v = -\nabla p - \nabla U_{E(t)} -m\nabla U_{X(t)}  & \text{in } E(t),
			\\
			\nabla\cdot v=0 & \text{in } E(t), \\
			n\cdot v= V_N & \text{on } \partial E(t), \\
		\end{cases}
	\end{align}
	where $ V_N $ is the normal velocity of the interface $ \partial E(t) $ and $ n $ the outer unit normal vector of $ \partial E(t) $. Here, $ U_{E(t)} $ and $ U_{X(t)} $ are the gravitational potentials, see below for the precise definitions.  Furthermore, $ p=p(t,x) $ is the scalar pressure which describes the internal pressure of the body for $ x\in E(t) $ and the external pressure of the surrounding space for $ x\in \R^2\backslash E(t)$. We assume the external pressure to be constant on $\R^2\backslash E(t) $ and without of generality we can take this constant to be zero. This reflects that the configuration is surrounded by a uniform medium. Therefore, the continuity of the pressure at the interface that separates the liquid from the exterior implies that
    \begin{align}\label{ext:pressure}
		p=0 \quad \text{on } \partial E(t).
	\end{align}

    Since there are no external forces acting on the configuration described by the fluid body and the external particle, their common center of mass moves at constant speed. Consequently, we can assume without loss of generality (using a change of the coordinate system) that the center of mass is at zero, i.e.
	\begin{align}\label{eq:mass:particle:center}
		\int_{E(t)}x\, dx +mX(t) = 0.
	\end{align} 
	
	As mentioned in the introduction we study two cases for the potentials $ U_{E(t)} $ and $ U_{X(t)} $ in \eqref{eq:IncompEulerPoissonPrev}.
	\begin{enumerate}[(A)]
		\item\label{CaseA} We consider a family of power law potentials, more precisely for $ \nu\in (0,1] $ we define
		\begin{align}\label{potentials3D}
			U_{X(t)}(x):= -\dfrac{1}{|x-X(t)|^\nu}, \quad U_ {E(t)}(x) := -\int_{E(t)} \dfrac{dy}{|x-y|^\nu}.
		\end{align} 
		\item\label{CaseB} We consider potentials given via the fundamental solution of the (two-dimensional) Laplace operator, i.e.
		\begin{align}\label{potentials2D}
			U_{X(t)}(x):= \ln|x-X(t)|, \quad U_ {E(t)}(x) := \int_{E(t)} \ln|x-y|\, dy.
		\end{align}
	\end{enumerate}
	Note that in both cases the signs are chosen to yield attractive forces. Furthermore, Case \eqref{CaseA} with $ \nu=1 $ can be interpreted as Newtonian gravitational interactions.
	
	Let us mention here that in Case \eqref{CaseA} with $ \nu=1 $ some care is needed in order to define a solution to \eqref{eq:IncompEulerPoissonPrev} since the gradient $ \nabla U_{E(t)} $ is not well-defined due to the onset of a singularity. However, this does not suppose a problem since the pressure gradient $\nabla p$ has also a similar singularity with a reverse sign that compensates the singularity of $\nabla U_{E(t)}$. In order to avoid this singular terms, it is convenient to rewrite the problem \eqref{eq:IncompEulerPoissonPrev} substracting the hydrostatic pressure. To this end, we define $ p = P-U_{E(t)}-mU_{X(t)}$ where $P$ is the non-hydrostatic pressure. Then the system \eqref{eq:IncompEulerPoissonPrev} turns into 
	\begin{align}\label{eq:IncompEulerPoissonPrev2}
		\begin{cases}
			\partial_t v + (v\cdot \nabla) v = -\nabla P  & \text{in } E(t),
			\\
			\nabla\cdot v=0 & \text{in } E(t), \\
			n\cdot v= V_N & \text{on } \partial E(t), \\
			P=U_{E(t)}+mU_{X(t)} & \text{on } \partial E(t),
		\end{cases}
	\end{align}
	where the last equation follows from \eqref{ext:pressure}. Now, these equations do not contain singular terms.
	
	The solutions to \eqref{eq:IncompEulerPoissonPrev2} studied in this paper are classical solutions, i.e. $ v:E(t)\to \R^2 $ and $ \partial E(t) $ are regular. However, the function $ P:\overline{E(t)}\to \R $ is in general only continuous, i.e. in Case \eqref{CaseA} the gradient $ \nabla P $ is not defined on $ \partial E(t) $. As we will see in the next section, this condition of continuity of the pressure and the last equation in \eqref{eq:IncompEulerPoissonPrev2} yields an equation for the free-boundary.
	
	Furthermore, the solutions constructed in this paper occur as perturbations of solutions to the time-independent equation with $ m=0 $, that is
	\begin{align}\label{eq:UnperIncompEulerPoisson}
		\begin{cases}
			(v\cdot \nabla) v = -\nabla P & \text{in } E,
			\\
			\nabla\cdot v=0 & \text{in } E, \\
			n\cdot v= 0 & \text{on } \partial E
			\\
			P = U_{E}  & \text{on } \partial E.
		\end{cases}
	\end{align}
	One particular solution we consider is given by the unit disk $ E=\D $ together with a corresponding velocity field $ v $ and the non-hydrostatic pressure $ P $.
	
	In addition, we assume that the perturbed fluid body and the external particle solving \eqref{eq:IncompEulerPoissonPrev} rotate around their center of mass with angular speed of rotation $ \Omega_0>0 $. Furthermore, we look for configurations which are time-independent in a rotating frame at angular speed $ \Omega_0 $, see Figure \ref{fig:picture}. Changing to such a rotating coordinate system we obtain the equations
	\begin{align}\label{eq:IncompEulerPoisson}
		\begin{split}
			\begin{cases}
				(v\cdot \nabla) v + 2\Omega_0  Jv - \Omega_0^2 x = -\nabla P & \text{in } E
				\\
				\nabla \cdot v=0 & \text{in } E
				\\
				n\cdot v=0 & \text{on } \partial E
				\\
				P=U_E+mU_X & \text{on } \partial E
				\\
				\Omega_0^2 X = \nabla U_E(X)
				\\
				|E|=\pi
				\\
				\int_{E}x\, dx +mX=0.
			\end{cases}
		\end{split}
	\end{align}
	In equations \eqref{eq:IncompEulerPoisson} we used the matrix $ J $ defined by
	\begin{align*}
		J = \left( \begin{matrix}
			0 & -1 \\ 1 & 0
		\end{matrix}  \right)
	\end{align*}
	which encodes the action of the vector product in the two-dimensional case.
	
	\begin{figure}[ht]
		\centering
		\includegraphics[width=0.7\linewidth]{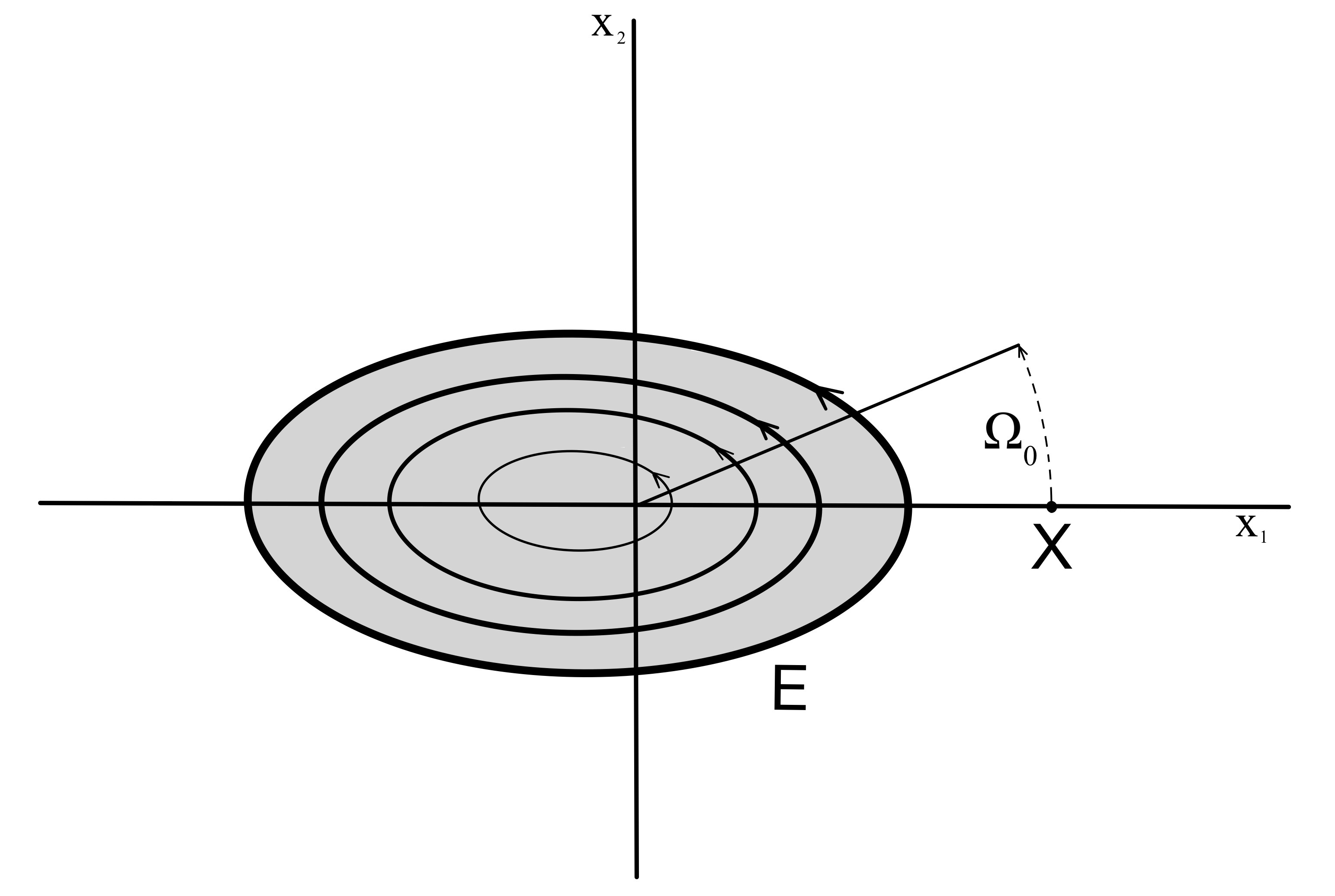}
		\caption{Configuration of the fluid body $ E $ and the external particle $ X $. Both rotate around their common center of mass (at the origin) with angular speed $ \Omega_0 $.}
		\label{fig:picture}
	\end{figure}
	
	Notice that in this setting, the shape of the body $ E $, the velocity field $ v $ and the position of the particle $ X $ do not depend on time. Furthermore, we construct solutions $ v\neq0 $, which can be interpreted as some type of tidal waves induced by the gravity of the external particle as well as the velocity of the unperturbed fluid.
	
	We briefly comment on the system of equations \eqref{eq:IncompEulerPoisson}. First, note that the terms $ 2\Omega_0  Jv $ and $ - \Omega_0^2 x $ represent the Coriolis and the centrifugal forces, respectively, which appear in the rotating frame of reference. The third equation in \eqref{eq:IncompEulerPoisson} ensures that the free-boundary is stationary, i.e. the fluid inside the body does not move across the boundary. As stated above, the external pressure is assumed to be constant outside the body. The equation $ \Omega_0^2 X = \nabla U_E(X) $ follows from Newton's law and ensures that the external particle is at rest. Note that $ \nabla U_E(X) $ is now well-defined also in Case \eqref{CaseA}, since we consider only cases with $ X $ separated from $ E $. The centrifugal force acting on $ X $ balances with the gravitational force of the fluid body. In addition, for definiteness, we assume that the total mass of the fluid is $ \pi=|\D| $. The last equation in \eqref{eq:IncompEulerPoisson} ensures that the center of mass is at the origin.
	
	Finally, let us mention that equations \eqref{eq:IncompEulerPoisson} are invariant under rotations around the origin. Hence, we can assume w.l.o.g. that the particle $ X=(a,0) $ is located on the $ x_1 $-axis. In particular, a solution to \eqref{eq:IncompEulerPoisson} yields a family of solutions by applying rotations.
	
	In this paper, we construct solutions to \eqref{eq:IncompEulerPoisson} obtained as perturbation of solutions to \eqref{eq:UnperIncompEulerPoisson} with $E=\D$ by means of an implicit function theorem in H\"older spaces. We require a non-resonance condition on $ \Omega_0 $ and a non-degeneracy condition on the unperturbed velocity field solving \eqref{eq:UnperIncompEulerPoisson}, see Theorem \ref{thm:MainThm} and Corollary \ref{cor:MainResult}.
	
	The paper is organized as follows. In Section \ref{sec:Reformulation} we reformulate the problem using Grad-Shafranov, the Bernoulli equation and conformal mappings to derive a reduced system of equations that will be more amenable to mathematical analysis. These new system is solved using an implicit function theorem. To this end, we provide some preliminary results concerning conformal mapping properties, estimates for elliptic equations, as well as suitable representations of the gravitational potentials in Section \ref{sec:preliminar}. In Section \ref{sec:Frechet} we prove the Fr\'{e}chet differentiability of the reduced system of equations w.r.t. the unknowns of the problem. Furthermore, we prove the invertibility of the Fr\'{e}chet derivative at the unperturbed solution in Section \ref{sec:inver}. Finally, we conclude the article with the proof of the main results in Section \ref{sec:main:proof}.
	
	\section{Reformulation of the problem and main result}\label{sec:Reformulation}
	In this section, we reduce the problem \eqref{eq:IncompEulerPoisson} to a set of equations that will be studied in the main part of the paper. To this end, we apply in particular conformal mappings as well as the Grad-Shafranov method.
	
	\subsubsection*{Conformal mappings}
	We use conformal mappings, i.e. bijective analytic functions, to parameterize the domain of the fluid. Recall that by the Riemann mapping theorem for any simply connected domain $ E\subset \C $ one can find a conformal mapping $ f:\D\to E $. Here, we identify $ \C $ with $ \R^2 $ via $ z=x_1+ix_2 $. In the case of smooth domains the mapping extends conformally to $ \overline{\D}\to \overline{E} $. 
	
	In our study, we consider conformal mappings of the form $ f_h:\D\to \R^2 $, $ f_h(z)=z+h(z) $, where $ h $ is small such that the domain is close to the disk. Let us mention that under a general smallness condition on some arbitrary analytic function $ h:\D\to \C $ the mapping $ f_h $ is conformal, see Lemma \ref{lem:ConformalMapping}. We denote the corresponding domain by $ E_h=f_h(\D) $ to emphasize the dependence on $ h $. Accordingly, we use the notation $ U_h=U_{E_h} $. Furthermore, we denote by $ f_h' $ the complex derivative, i.e. understanding $ f_h $ as a mapping $ \D\subset \C\to \C $.
	
	Let us also introduce the so-called Blaschke factors, see \cite{Rudin1987ComplexAna}, defined by
	\begin{align}\label{eq:Blaschke}
		b_{c,d}(z) = d \dfrac{z-c}{1-\overline{c}z}, \quad c\in \D, d\in \C, |d|=1.
	\end{align}
	These factors are the only conformal mappings $ \D\to \D $. Choosing $ c,\, d $ accordingly allows to set $ h(0)=0 $ and $ h'(0)\in \R $ by replacing $ f_h $ by $ f_h\circ b_{c,d} $. This defines the conformal mapping $ f_h $ and hence also $ h $ uniquely.
	
	\subsubsection*{Grad-Shafranov method}
	In order to construct the velocity field $ v $ solving \eqref{eq:IncompEulerPoisson} we use the Grad-Shafranov method. For convenience we recall the approach of this tool. Let us denote by $\omega=\nabla\times v$ the vorticity of the velocity field $ v $. Here, we identified the velocity fields $ v=(v_1,v_2) $ and $ (v_1,v_2,0) $ for the vector product. Furthermore, we use the notation $ \Omega = (0,0,\Omega_0) $. Accordingly, the vorticity is a vector field, which has the form $ \omega=(0,0,\omega_3) $. We will abuse notation by writing $ \omega $ for both the vector field and its only non-zero component $ \omega_3 $.
	
	By applying the vector identity $ -\omega Jv = v\times \omega=-(v\cdot \nabla)v+\frac{1}{2}\nabla(|v|^{2}) $ we infer that the first three equations in \eqref{eq:IncompEulerPoisson} can be written as
	\begin{align} \label{eq:VorticityFormulation}
		\begin{cases}
			-(\omega+ 2\Omega_0) J v=  v\times (\omega+2\Omega) = \nabla H, & \text{in } E_h 
			\\
			\nabla \cdot v=0, & \text{in } E_h \\
			n_h\cdot v=0, & \text{on } \partial E_h.
		\end{cases}
	\end{align}
	Here, $ H $ is called the Bernoulli head and is defined by
	\begin{align*}
		H:= P+ \dfrac{1}{2}|v|^2 - \dfrac{\Omega_0^2}{2}|x|^2.
	\end{align*}
	The term $ 2\Omega_0 $ can be interpreted as the vorticity of the velocity field $ \Omega \times x = \Omega_0 Jx $ which occurs in terms of the Coriolis force due to the rotating frame of reference. Applying the operator $ \nabla\times $ to the first equation in \eqref{eq:VorticityFormulation} and using that $ \nabla\cdot v=\nabla\cdot \omega=0 $ yields
	\begin{align*}
		(v\cdot \nabla) (\omega+2\Omega)=0.
	\end{align*}
	Let us remark that this identity holds in general only in two dimensions, which restricts the Grad-Shafranov method to these situations. As a corollary of the above identity we obtain that $ \omega+2\Omega_0 $ and thus $ \omega $ is constant along stream lines (characteristics) of $ v $.
	
	The main object in the Grad-Shafranov approach is the stream function $ \psi:E_h\to \R $ satisfying $ v=\nabla^\perp \psi:=J\nabla \psi =  (-\partial_{x_2}\psi,\partial_{x_1}\psi) $. Let us mention that in general the existence of a function $ \psi $ is ensured in two dimensions by the divergence-free condition on $ v $ as well as topological properties on $ E_h $. Here, $ E_h $ is simply connected.
	
	Now, with the stream function at hand, we can write $ \omega = \Delta \psi $. Since $ \psi $ is also constant along the characteristics of $ v=J\nabla \psi $, one might conclude the existence of a function $ G:\R\to \R $ such that $ \Delta \psi = G(\psi) $. Let us remark here that in general the existence of $ G $ can be concluded only locally when $ \nabla \psi \neq 0 $. Furthermore the function $ G $ might be multi-valued, a situation, although interesting we will not consider in this paper. In addition, we require $ n_h\cdot v = 0 $ on $ \partial E_h $, and thus 
	\begin{align*}
		0=n_h\cdot J\nabla\psi= \tau_h\cdot \nabla\psi,
	\end{align*}
	where $ \tau_h $ is the positively-oriented tangential vector on $ \partial E_h $. We integrate along the boundary to get for $ x\in \partial E_h $ that $ \psi(x) = c_0 $ for some constant $ c_0\in\R $. Note that the potential $ \psi $ is given up to a constant, so we can choose $ c_0=0 $ by adapting the function $ G $ if needed. Thus, the stream function solves the equation
	\begin{align}\label{eq:StreamFunction}
		\begin{cases}
			\Delta \psi = G\left( \psi \right)  & \text{in } E_h,
			\\
			\psi = 0 & \text{on } \partial E_h,
		\end{cases}
	\end{align}
	In the Grad-Shafranov approach the above reasoning is reversed in the sense that we are given some (regular enough) function $ G $ and we construct the stream function (hence also the velocity field) by solving the equation \eqref{eq:StreamFunction}.
	
	Note that $ \psi:E_h\to \R $ is a function of $ h $, so that we sometimes write $ \psi_h $ if we want to emphasize the dependence in $h$. Let us also remark that the existence and uniqueness of solutions to \eqref{eq:StreamFunction} is ensured in general by assuming that $ G $ is non-decreasing, see Lemma \ref{lem:WellPosedStreamFunction}.
	
	We now use the conformal mapping in order to reduce equation \eqref{eq:StreamFunction} to the domain $ \D $. We set $ \phi_h:=\psi_h \circ f_h $, which is now defined on the disk $ \phi_{h}:\D\to \R $. The corresponding equation reads
	\begin{align}\label{eq:StreamFunctionDisk}
		\begin{cases}
			\Delta \phi_h = |f_h'|^2G\left( \phi_h \right)  & \text{in } \D,
			\\
			\phi_h = 0 & \text{on } \partial\D.
		\end{cases}
	\end{align}	
	
	\subsubsection*{Equation of the free-boundary} 
	The equation determining the free-boundary can be derived from the fact that the non-hydrostatic pressure $ P $ is continuous along the free-boundary. We can write using the stream function $ v=\nabla^\perp \psi_h $
	\begin{align*}
		v\times (\omega+2\Omega) = -(G(\psi_h)+2\Omega_0) J\nabla^\perp\psi_h =  (G(\psi_h)+2\Omega_0) \nabla\psi_h, \quad \text{in } E_{h}.
	\end{align*}
	We conclude that 
	\begin{align*}
		v\times (\omega+2\Omega) = \nabla \left[  F\left( \psi_h \right) \right] , \quad F(\psi_h)\mid_{\partial E_h} = 0, 
	\end{align*}
	where $ F'=G+2\Omega_0 $ is a primitive with $ F(0)=0 $. Consequently, in order to ensure equality in the first equation in \eqref{eq:VorticityFormulation} the non-hydrostatic pressure is given (up to a constant $ \lambda $) by
	\begin{align}\label{eq:ReconstructionPressure}
		P = F(\psi_h) - \dfrac{1}{2}|\nabla \psi_h|^2 + \dfrac{\Omega_0^2}{2}|x|^2 +\lambda \quad \text{in } E_h.
	\end{align}
	The condition that $ P $ is continuous along the free-boundary yields with $ P=U_h+mU_X $ on $ \partial E_h $ and $ F(\psi_h)\mid_{\partial E_h} = 0 $ the equation
	\begin{align}\label{eq:FreeBoundaryEq}
		\dfrac{1}{2}|\nabla \psi_h|^2 - \dfrac{\Omega_0^2}{2}|x|^2 +U_h +mU_{X} = \lambda \quad \text{on } \partial E_h.
	\end{align}
	The evaluation at the boundary $ \partial E_h=f_h(\partial \D) $ in \eqref{eq:FreeBoundaryEq} can be performed using the conformal mapping $ f_h $. We now summarize the reduced system that we aim to solve in our study
	\begin{align}\label{eq:FullSystem}
		\begin{cases}
			\dfrac{1}{2}\dfrac{|\nabla\phi_h|^2}{|f_h'|^{2}} - \dfrac{\Omega_0^2}{2}|f_h|^2 + U_h\circ f_h + mU_{X}\circ f_h  = \lambda & \text{on } \partial \D,
			\\
			\Delta \phi_h = |f_h'|^2G\left( \phi_h \right)  & \text{in } \D,
			\\
			\phi_h = 0 & \text{on } \partial\D.
			\\
			\Omega_0^2 a = \partial_{x_1} U_h(X)
			\\
			|E_h|=\pi.
		\end{cases}
	\end{align}
	Recall that the position of the particle is chosen as $ X=(a,0) $. The unknown triplet is $ (h,a,\lambda) $. As we will see, cf. Corollary \ref{cor:MainResult}, solutions of \eqref{eq:FullSystem} constructed in this paper yield solutions to \eqref{eq:IncompEulerPoisson}. Let us mention that the fourth equation in \eqref{eq:FullSystem} is the $x_{1}$-component of Newton equation for the particle $ X $, see also the fifth equation in \eqref{eq:IncompEulerPoisson}. The other component follows, as we will see in Corollary~\ref{cor:MainResult}, by the symmetry of the domain $ E $ w.r.t. the $ x_1 $-axis.
	
	\subsubsection*{Solution for $ m=0 $} 
	In the case when no external particle is present, i.e. $ m=0 $, we assume that the fluid body has the shape of a disk $ \D $. Furthermore, we consider a velocity field on $ \D $ with stream function $ \phi_0 $ solving
	\begin{align}\label{eq:UnperturbedStreamFunction}
		\begin{cases}
			\Delta \phi_0 = G\left( \phi_0 \right)  & \text{in } \D,
			\\
			\phi_0 =0 & \text{on } \partial \D.
		\end{cases}
	\end{align}
	Note that this coincides with \eqref{eq:StreamFunctionDisk} for $ h=0 $. Observe that due to the rotational invariance $ \phi_0=\phi_0(|x|)$ the equation reduces to the ODE
	\begin{align*}
		\dfrac{1}{r}\left( r\phi_0'\right)' = G\left(\phi_0(r)\right) , \quad \phi_0(1)=0.
	\end{align*} 
	This ODE is complemented with the condition that $ \lim_{r\to0}\phi_0(r) $ exists. Therefore, the velocity field becomes $ v(x)=\frac{\phi_0'(|x|)}{|x|}Jx $. It describes a non-uniform rotation with angular speed depending on the distance to the center. Since the velocity field is rotationally symmetric, the velocity in the non-rotating coordinate system is given by $ (\frac{\phi_0'(|x|)}{|x|}+\Omega_0)Jx $. Furthermore, note that the function $ \phi_0 $ can be extended to $ r>1 $. This is necessary, for instance, when evaluating $ \phi_0 $ on the boundary $ \partial E_h $, which is close to $ \partial \D $.
	
	The position of the unperturbed particle is chosen of the form $ X_0=(a_0,0) $. Since we consider only cases for which the fluid body and the external particle are strictly separated, we assume say $ a_0\geq2 $. Hence, $ E_h $ does not contain $ X\approx X_0 $ for small enough $ h $. The Newton equation for the particle requires that
	\begin{align*}
		\Omega_0^2X_0 = \nabla U_0(X_0).
	\end{align*}
	Further information of the potentials $ U_0 $ of the disk in both Case \eqref{CaseA} and Case \eqref{CaseB} are given in Lemma \ref{lem:UnperturbedPotential} and \ref{lem:PropertiesUnpertPot}. For $ a_0>1 $ we have $ U'_0(a_0)>0 $ and furthermore $ U'_0(a_0)/a_0\to 0 $ as $ a_0\to \infty $. In addition $ a_0\mapsto U'_0(a_0)/a_0 $ is strictly decreasing for $ a_0>1 $. Hence, there is a one-to-one correspondence between $ \Omega_0\in (0,\sqrt{U_0'(1)}] $ and $ a_0\geq1 $ via
	\begin{align}\label{eq:RelationPotentialParticlePosition}
		\Omega_0 = \sqrt{\dfrac{U_0'(a_0)}{a_0}}.
	\end{align}
	All in all, this defines a map $ \Omega_0\mapsto a_0(\Omega_0) $. Finally, the constant in \eqref{eq:FreeBoundaryEq} is given by $ \lambda_0= \frac{1}{2}\phi_0'(1)^2-\frac{1}{2}\Omega_0^2+U_0(1) $.
	
	\subsection{Notation}
	We will use the following notation throughout the manuscript.
	\begin{itemize}
		\item We use $\D$ to denote the unit disk with boundary $\partial \D$ and $\mathbb{T}=[0,2\pi]$ the $
		2\pi$-periodic torus with endpoints identified.  
		\item 
		The H\"{o}lder seminorm of a function $ u:\T\to \R $ or $ u:\D\to \R $ is defined by
		\begin{align*}
			\semnorm[k,\alpha]{u} &= \sup_{x_1\neq x_2}\dfrac{|u^{(k)}(x_2)-u^{(k)}(x_1)|}{|x_2-x_1|^\alpha}, \quad \alpha\in (0,1),
			\\
			\semnorm[k,0]{u} &= \norm[\infty]{u^{(k)}}, \quad \alpha=0.
		\end{align*}
		\item We abbreviate $ H^{k,\alpha}:=H^{k,\alpha}(\D):=H(\D)\cap C^{k,\alpha}(\overline{\D}) $, where $ H(\D) $ is the space of analytic functions on $\D $ and $ k\in \N_0 $, $ \alpha\in (0,1) $. We equip it with the standard H\"{o}lder norm $ \norm[k,\alpha]{\cdot} $.
		\item We denote by $ H^{k,\alpha}_0\subset H^{k,\alpha} $ the subspace of analytic functions $ h $ such that $ h(0)=0 $ and $ h'(0)\in \R $.
		\item Furthermore, the Fourier coefficients of a function $ g:\T\to \R $ are given by
		\begin{align*}
			\hat{g}_n = \dfrac{1}{2\pi} \int_0^{2\pi}g(\varphi) e^{in\varphi}\, d\varphi.
		\end{align*}
		Recall that $ \hat{g}_n = \overline{\hat{g}_{-n}} $, since $ g $ is real-valued.
		\item We denote by $ C_0^{k,\alpha}(\T)\subset C^{k,\alpha}(\T) $ those functions $ g $ with zero average, i.e. $ \hat{g}_0=0 $.
		\item Let us abbreviate with $ B_{r}=B_{r}(0)\subset H_0^{k,\alpha} $ the ball of radius $ r $ around zero. 
		\item We will denote with $C$ a positive generic constant that depends only on fixed parameters including $ \Omega_0 $ and norms of the function $ G $ in \eqref{eq:FullSystem}. Note also that this constant might differ from line to line.
	\end{itemize}
	
	\subsection{Main result and strategy towards the proof}
	In order to construct the desired solution, we make use of the implicit function theorem, cf. Lemma \ref{lem:IFT}. To do so, let us introduce the following functional spaces
	\begin{align}\label{eq:DefSpaces}
		\mathbb{X}^{k+2,\alpha}:=H_0^{k+2,\alpha}(\D)\times\R\times \R,
		\quad
		\mathbb{Z}^{k+1,\alpha}:= C^{k+1,\alpha}(\T)\times\R\times \R.
	\end{align}
	We define the following function related to the system \eqref{eq:FullSystem}. Define the map $ \F: U\times V\to \mathbb{Z}^{k+1,\alpha} $, where $ U\subset H_0^{k+2,\alpha}(\D)\times\R\times \R $, $ V\subset \R $, with $ X=(a,0) $,
	\begin{align}\label{eq:FullFunction}
		\F(h,a,\lambda,m) = \left( \begin{matrix}
			\restr{\left[ \dfrac{1}{2}\dfrac{|\nabla\phi_h|^2}{|f_h'|^{2}} - \dfrac{\Omega_0^2}{2}|f_h|^2 + U_h\circ f_h + mU_{X}\circ f_h  - \lambda \right]}{z=e^{i\varphi}}
			\\
			\Omega_0^2a-\partial_{x_1} U_h(X) 
			\\
			|f_h(\D)| -\pi
		\end{matrix}\right) 
	\end{align}

	The subset $ U $ is a sufficiently small neighborhood of $ (0,a_0,\lambda_0) $. In particular, it ensures that $ h $ defines a conformal mapping $ f_h(z)=z+h(z) $, see Lemma \ref{lem:ConformalMapping}. 
	
	Our goal is to solve the equation $ \F(h,a,\lambda,m)=0 $ via the implicit function theorem. To this end, we study the Fr\'echet derivative at the point $ (0,a_0,\lambda_0,0) $. We will apply a Fourier decomposition for the first component of $ \F $, which is a function on the torus $ \T $. As we will see, cf. Lemma \ref{lem:LinearizationZero}, the corresponding linear operator can be diagonalized and the Fourier multipliers have the form
	\begin{align}\label{eq:FourierMultipliers}
		\omega_n = -\dfrac{1}{2}\Omega_0^2 -\dfrac{1}{2}\phi_0'(1)^2(|n|+1)+\phi_0'(1)A_{|n|}'(1) (|n|+1)+c_{|n|}.
	\end{align} 
	The coefficients $ \omega_n $ are visible in a non-resonance condition for $ \Omega_0 $ in our main result, cf. Theorem \ref{thm:MainThm}. In the definition of $ \omega_n $ the function $ \phi_0 $ is the unperturbed stream function for $ m=0 $. The coefficients $ c_n $ enter through the interaction potential $ h\mapsto (U_h\circ f_h)(e^{i\varphi}) $. In Case \eqref{CaseA}, they are given by (note we identify again $ \R^2\backsimeq \C $)
	\begin{align}\label{eq:GravPotLinCoeffA}
		c_n = \dfrac{1}{2}\int_{\D} \left( \nu\dfrac{1-y^{n+1}}{1-y}-2(n+1)y^{n} \right) \dfrac{dy}{|1-y|^\nu},
	\end{align}
	and in Case \eqref{CaseB} by
	\begin{align}\label{eq:GravPotLinCoeffB}
		c_n=\begin{cases}
			\frac{\pi}{2}\left(1-\frac{1}{n}\right) & n\geq 1,
			\\
			\frac{\pi}{2} & n=0.
		\end{cases}
	\end{align}
	Let us note that the integral in \eqref{eq:GravPotLinCoeffA} defines a real quantity.
	
	Finally, the numbers $ A_n'(1) $ are computed by means of the functions $ A_n:(0,1)\to \R $ solving the ODE
	\begin{align}\label{eq:LinearizationStreamFunctODE}
		\dfrac{1}{r}(rA_n')' -\dfrac{n^2}{r^2}A_n-G'(\phi_0(r))A_n = r^{|n|}G(\phi_0(r)), \quad A_n(1)=0.
	\end{align}
	They appear in the Fr\'echet derivative of the stream function $ h\mapsto \phi_h $, cf. Section \ref{sec:inver}.
	
	The main result of this work reads as follows.
	\begin{thm}\label{thm:MainThm}
		Let $ k\in \N_0 $, $ \alpha\in (0,1) $ and $ a_0\geq 2 $. Assume $ G\in C^{k+3}(\R;\R) $ to be non-decreasing. Let $ \Omega_0\geq0 $ be related to $ X_0=(a_0(\Omega_0),0) $ as stated in \eqref{eq:RelationPotentialParticlePosition} and let the non-resonance condition
		\begin{align}\label{eq:NonResonanceAssumption}
			\forall n\in \N \, : \, \omega_n\neq 0
		\end{align}
		be satisfied for $ \omega_n $ given in \eqref{eq:FourierMultipliers}. Furthermore, we assume for the unperturbed stream function $ \phi_0 $ that
		\begin{align}\label{eq:ConditionUnperturbedStreamFunction}
			\phi'_0(1)\neq 0.
		\end{align}
		
		Then, there are $ \delta>0 $, $ \varepsilon>0 $ such that for any $ m\in [0,\delta) $ there is a unique solution $ (h,a,\lambda)\in \mathbb{X}^{k+2,\alpha} $ of the equation $ \F(h,a,\lambda,m)=0 $ satisfying
		\begin{align*}
			\norm[k+2,\alpha]{h} + |a-a_0|+|\lambda-\lambda_0|< \varepsilon.
		\end{align*}
		Furthermore, the dependence $ m\mapsto (h,a,\lambda)(m) $ is continuous.
	\end{thm}
	As a corollary we obtain that a solution to $ \F(h,a,\lambda,m)=0 $ yields a solution to our original problem \eqref{eq:IncompEulerPoisson}.
	\begin{cor}\label{cor:MainResult}
		Under the assumption of Theorem \ref{thm:MainThm}, the domain $ E_h=f_h(\D) $ in Theorem~\ref{thm:MainThm} is symmetric w.r.t. the $ x_1 $-axis. Finally, the corresponding velocity field $ v=\nabla^\perp\psi_h $ together with the position of the particle $ X=(a,0) $ and the non-hydrostatic pressure $P$ yield a solution to \eqref{eq:IncompEulerPoisson}.
	\end{cor}
	\begin{rem}\label{rem:NonResonanceCondition}
		Let us comment on the non-resonance condition \eqref{eq:NonResonanceAssumption}. 
		\begin{enumerate}[(i)]
			\item It ensures that the linearized operator can be inverted in order to apply the implicit function theorem. In the case that \eqref{eq:NonResonanceAssumption} is not satisfied bifurcations to other shapes might occur. 
			\item Note that the condition \eqref{eq:NonResonanceAssumption} is needed only for $ n\in \N $, since $ \omega_n $ only depends on $ |n| $. Furthermore, as we will see in Lemma \ref{lem:AsymptotcisCoeff} and Lemma \ref{lem:AsymptotcisCoeffGravPot} the leading order term on the right hand side of \eqref{eq:NonResonanceAssumption} is given by $ -\phi_0'(1)^2(|n|+1) $, whereas the other terms are at most of order $ \mathcal{O}(\ln n) $ as $ n\to \infty $. In particular, the condition \eqref{eq:NonResonanceAssumption} is automatically satisfied for sufficiently large $ n $. Hence, it is possible to verify the condition numerically.
			\item In the particular case that the fluid has no internal motion in the non-rotating coordinate system for $ m=0 $ we have $ v(x)=-2\Omega_0Jx $ and thus $ \phi_0(x)= -\Omega_0|x|^2/2 $. This corresponds to the choice $ G= -2\Omega_0 $. Then, we can readily check that solutions to \eqref{eq:LinearizationStreamFunctODE} have the form
			\begin{align*}
				A_n(r)&=-2\Omega_0\dfrac{r^n(r^2-1)}{4n+4}, \quad A_n'(1)=-\dfrac{\Omega_0}{n+1}.
			\end{align*}
			Hence, the condition \eqref{eq:NonResonanceAssumption} reduces to
			\begin{align*}
				\omega_n=-\dfrac{|n|}{2}\Omega_0^2+c_{|n|}\neq 0.
			\end{align*}
		\end{enumerate}
	\end{rem}
	\begin{rem}
		Let us mention that the assumption \eqref{eq:ConditionUnperturbedStreamFunction} in Theorem \ref{thm:MainThm} is also needed to prove the invertibility of the Fr\'echet derivative in order to apply the implicit function theorem. This condition implies that the function $ \phi_0 $ has no local extremum at the boundary. When perturbing such extrema, saddle points are created generically. Consequently, vortices would appear.
	\end{rem}
    \begin{rem}
        We are assuming in Theorem \ref{thm:MainThm} that $m\geq 0$ since it is the most natural setting from the physical point of view. However, the proof of Theorem \ref{thm:MainThm} is also valid for the case $m\in(-\delta,\delta)$. 
    \end{rem}
	\begin{rem}
		In this paper we restricted ourselves to interaction potentials defined in Case \eqref{CaseA} and Case \eqref{CaseB}. The study of more general interactions, would require further modifications. In particular, a better understanding of results like Lemma \ref{lem:PeriodicPseudoDiffReg} on pseudo-differential operators on the torus.
	\end{rem}
	
	\section{Preliminary results}\label{sec:preliminar}
	We collect here some auxiliary results that will be used in the subsequent sections. Let us start with a well-known result in complex analysis regarding analytic functions.
	\begin{lem}\label{lem:ConformalMapping}
		Consider the analytic function $ f_h(z)=z+h(z) $ with $ \norm[C^1(\overline{\D})]{h}<1/\sqrt{2} $. Then, $ f_h:\D\to f_h(\D) $ is conformal.
	\end{lem}
	\begin{proof}
		We prove that $ f_h $ is injective. Define the function $ \zeta(\varphi)=f_h(e^{i\varphi}) $, $ \varphi\in\T $. Let $ \varphi_1,\, \varphi_2\in \T $. We can assume $ |\varphi_1-\varphi_2|\leq \pi $. If $ |\varphi_1-\varphi_2|\geq \pi/2 $ we have
		\begin{align*}
			\left| \zeta(\varphi_2)-\zeta(\varphi_1) \right| \geq  \left| e^{i\varphi_2}-e^{i\varphi_1}\right| -2\norm[C(\overline{\D})]{h}= 2\left| \sin \left( \dfrac{\varphi_2-\varphi_1}{2} \right) \right| -2\norm[C(\overline{\D})]{h}>0.
		\end{align*}
		On the other hand,  if $ |\varphi_1-\varphi_2|< \pi/2 $ we estimate
		\begin{align*}
			\left| \zeta(\varphi_2)-\zeta(\varphi_1) \right| &= \left| \int_{\varphi_1}^{\varphi_2}f_h'(e^{i\psi})ie^{i\psi}\, d\psi \right| \geq \left| e^{i\varphi_2}-e^{i\varphi_1}\right|  -|\varphi_2-\varphi_1|\norm[C^1(\overline{\D})]{h}
			\\
			&= 2\left| \sin \left( \dfrac{\varphi_2-\varphi_1}{2} \right) \right| -|\varphi_2-\varphi_1|\norm[C^1(\overline{\D})]{h} \geq |\varphi_2-\varphi_1| \left( \dfrac{1}{\sqrt{2}}-\norm[C^1(\overline{\D})]{h} \right).
		\end{align*}
		Hence, $ f_h $ is one-to-one on the boundary. As a consequence of the Darboux-Picard theorem, see \cite[Thm. 9.16]{Burckel1979CompAna}, $ f_h $ is injective on $ \overline{\D} $.
	\end{proof}
	
	\begin{lem}[Fa\`{a} di Bruno formula \cite{FaaDiBruno}]\label{lem:FaaDiBruno} For any $ n\in \N $ and two functions $ f, \, g \in C^n(\R;\R) $ we have the formula
		\begin{align*}
			\dfrac{d^n}{dx^n}(f\circ g)(x) = \sum_{\substack{\ell_1,\ldots, \ell_n
					\\ 1\cdot \ell_1+\cdots +n\cdot \ell_n=n}} n! \,  \left[ \dfrac{d^{\ell_1+\cdots +\ell_n}}{dx^{\ell_1+\cdots+ \ell_n}}f \right]  (g(x)) \, \prod_{j=1}^n \left( \dfrac{1}{\ell_j! \, j!}\dfrac{d^j}{dx^j}g(x) \right)^{\ell_j}.
		\end{align*}
	\end{lem}
	We recall the following version of the implicit function theorem.

	\begin{lem}[Implicit function theorem, \cite{Deimling}]\label{lem:IFT}
		Let $ \mathsf{X},\, \mathsf{Y},\, \mathsf{Z} $ be Banach spaces and $ U\subset \mathsf{X} $, $ V\subset \mathsf{Y} $ be neighborhoods of $ x_0,\, y_0 $, respectively, where $ \F(x_0,y_0)=0 $. Suppose that $ \F: U \times V\to \mathsf{Z} $ is continuous, continuously differentiable with respect to $ x\in U $ and $ D_x\F(x_0,y_0)\in \mathcal{L}(\mathsf{X},\mathsf{Z}) $ is invertible. Then, there are balls $ B_\varepsilon(x_0)\subset U $, $ B_\delta(y_0)\subset V $ and a unique map $ \xi:B_\delta(y_0)\to B_\varepsilon(x_0) $ with $ \F(\xi(y),y)=0 $ for all $ y\in B_\delta(y_0) $. Furthermore, $ \xi $ is continuous.
	\end{lem}
	Here, we denote by $ \mathcal{L}(\mathsf{X},\mathsf{Z}) $ the space of bounded linear operators $ \mathsf{X}\to\mathsf{Z} $. Furthermore, $ D_x\F(x_0,y_0)\in \mathcal{L}(\mathsf{X},\mathsf{Z}) $ is the Fr\'echet derivative w.r.t. the first variable, i.e. we have
	\begin{align*}
		\F(x_0+\xi,y_0)=\F(x_0,y_0) + D_x\F(x_0,y_0)[\xi] + o(\norm[\mathsf{X}]{\xi})
	\end{align*}
	as $ \norm[\mathsf{X}]{\xi}\to 0 $.
	
	Let us also give an existence and uniqueness result for the equation \eqref{eq:StreamFunctionDisk}. Such elliptic equations have been studied extensively both in H\"older and Sobolev spaces, see e.g. \cite{Giaquinta2013Regularity,GilbargTrudinger2001}.
	\begin{lem}\label{lem:WellPosedStreamFunction}
		Let $ h\in B_{1/2}\subset H_0^{k+2,\alpha} $ and assume $ G\in C^{k+3}(\R;\R) $ to be non-decreasing. Then there is a unique solution $ \phi_h\in C^{k+2,\alpha}(\overline{\D}) $ to \eqref{eq:StreamFunctionDisk}. Furthermore, there exists a constant $ C>0 $ independent of $ h $ such that
		\begin{align}\label{eq:WellPosedStreamFunctionEstimate}
			\norm[C^{k+2,\alpha}(\overline{\D})]{\phi_h}\leq C.
		\end{align}
	\end{lem}
	\begin{proof}
		We prove the assertion in terms of $ \psi_h=\phi_h\circ f_h^{-1} $. The existence follows from standard methods of calculus of variations applied to the functional
		\begin{align*}
			\psi\mapsto \int_{E_h} |\nabla\psi|^2\, dx + \int_{E_h}F\left( \psi \right) \, dx,
		\end{align*}  
	    where $ F'=G $ is a primitive. Note that $ F $ is convex, since $ G $ is non-decreasing. The regularity follows via a bootstrapping argument, recalling that $ G\in C^{k+3}(\R;\R)$. Observe that due to $ f_h\in H^{k+2,\alpha} $, the boundary $ \partial E_h $ is sufficiently regular. The uniqueness can be proved using a comparison principle, since $ G $ is non-decreasing.
		
		The estimate \eqref{eq:WellPosedStreamFunctionEstimate} is a consequence of the maximum principle and Schauder estimates. Indeed, this will be done by separating two cases.
		
		\textit{Case 1.} We assume that there is $ y_0\in \R $ with $ G(y_0)=0 $. Since $ G $ is non-decreasing, we can find $ N>0 $ sufficiently large such that $ G\left( -N \right)\leq 0\leq G\left( N \right) $. We conclude from a comparison principle that $ \norm[\infty]{\phi_h}\leq N $. Hence, the right-hand side in \eqref{eq:StreamFunctionDisk} is uniformly bounded in $ h $. We apply regularity theory in Sobolev spaces to conclude that $ \phi_h\in W^{2,2} $ with a bound independent of $ h\in B_{1/2} $. Hence, by Sobolev embedding we obtain $ \phi_h\in C^\alpha $. Now, the right-hand side in \eqref{eq:StreamFunctionDisk} is uniformly bounded in $ C^\alpha $. We hence apply repeatedly Schauder estimates to yield the result.
		
		\textit{Case 2.} If $ G $ is always non-zero, we can assume w.l.o.g. that $ G>0 $. In this case, we infer $ \phi_h\leq 0 $ by the maximum principle. Thus, the right-hand side in \eqref{eq:StreamFunctionDisk} is uniformly bounded. We can now argue as in Case 1.
	\end{proof}
	
	Next, we prove a formula for the unperturbed interaction potential of $\D$, i.e. of the unperturbed solution for $ m=0 $, in Case \eqref{CaseA} with $ \nu=1 $ and Case \eqref{CaseB}.
	\begin{lem}\label{lem:UnperturbedPotential}
		The following formulas hold in Case \eqref{CaseA} with $ \nu=1 $
		\begin{align}\label{eq:UnperturbedPotentialA}
			\begin{split}
				U_0(r) &= -\dfrac{4}{\pi^2} \sum_{k\geq 0}W_{2k}^2 \left( \dfrac{r}{2k+2}+\dfrac{r}{2k-1}-\dfrac{r^{2k}}{2k-1} \right), \quad 0\leq r\leq 1,
				\\
				U_0(r) &= -\dfrac{4}{\pi^2} \sum_{k\geq 0}\dfrac{W_{2k}^2 }{2k +2} \dfrac{1}{r^{2k+1}}, \quad r \geq 1.
			\end{split}
		\end{align}
		Here, $ W_\ell= \frac{\pi}{2}\frac{(\ell-1)!!}{\ell!!}  $ is Wallis' formula, see \cite[Formula 6.1.49]{Abramowitz1965Handbook}.
		In Case \eqref{CaseB} we have that
		\begin{align}\label{eq:UnperturbedPotentialB}
			U_0(r) = \begin{cases}
				-\dfrac{\pi}{2}(1-r^2) & r\leq 1,\\
				\pi \ln r & r\geq 1.
			\end{cases}
		\end{align}
	\end{lem}
	Let us recall that $ \lim_{\ell\to \infty} \sqrt{\ell}W_\ell = \sqrt{\pi/2} $. Consequently, the series in \eqref{eq:UnperturbedPotentialA} converges also for the critical value $ r=1 $.
	\begin{proof}
		To this end, we use a multipole expansion for $ x=x(r,\theta,\varphi) $, $  y=y(s,\theta',\varphi') \in \R^3 $
		\begin{align*}
			\dfrac{1}{|x-y|} &= \dfrac{1}{\sqrt{r^2+s^2-2rs(\cos\theta \cos\theta' + \sin\theta \sin\theta' \cos(\varphi-\varphi'))}}
			\\
			&= \sum_{\ell\geq 0}\sum_{|m|\leq \ell} \dfrac{4\pi}{2\ell+1} \dfrac{(r\wedge s)^\ell}{(r\vee s)^{\ell+1}} Y_{\ell,m}(\theta,\varphi)Y_{\ell,m}(\theta',\varphi')^*,
		\end{align*}
		where the spherical harmonics are given by
		\begin{align*}
			Y_{\ell,m}(\theta,\varphi) &= \sqrt{\dfrac{2\ell+1}{4\pi} \dfrac{(\ell-m)!}{(\ell+m)!}} e^{im\varphi}P_\ell^m (\cos\theta).
		\end{align*}
		Here, $P_{\ell}^{m}$ are the associated Legendre polynomials. We have for $ \theta=\theta'=\pi/2 $
		\begin{align*}
			U_0(r) = -\sum_{\ell\geq 0}\dfrac{4\pi c_\ell^2}{2\ell+1} \, \int_0^1 \dfrac{(r\wedge s)^\ell \, sds}{(r\vee s)^{\ell+1}}, \quad c_\ell := Y_{\ell0}(\pi/2,0) = \sqrt{\dfrac{2\ell+1}{4\pi}}P_\ell(0).
		\end{align*}
		
		A computation shows that
		\begin{align*}
			P_\ell(0) = \begin{cases}
				(-1)^{\ell/2}\dfrac{(\ell-1)!!}{\ell!!} & \ell \text{ even},
				\\
				0 & \ell \text{ odd}.
			\end{cases}
		\end{align*}
		Rewriting the coefficients of the series in terms of the Wallis' formula and choosing $ \ell=2k $ yields both formulas in \eqref{eq:UnperturbedPotentialA}. The formula in \eqref{eq:UnperturbedPotentialB} follows by solving the Poisson equation $ \Delta U_0 = \ind_\D $.
	\end{proof}
	The following lemma contains information on the unperturbed potential $ U_0 $ in all cases considered.
	
	\begin{lem}\label{lem:PropertiesUnpertPot}
		The potential $ U_0 $ satisfies in Case \eqref{CaseA} and \eqref{CaseB} respectively
		\begin{enumerate}[(i)]
			\item $ U_0'(r)>0 $ for $ r>1 $,
			\item $ U_0''(r)<0 $ for $ r>1 $,
			\item $ r\mapsto U_0'(r)/r $ is strictly decreasing for $ r\geq1 $ and $ \lim_{r\to \infty} U_0'(r)/r=0 $.
 		\end{enumerate}
	\end{lem}
	\begin{proof}
		In Case \eqref{CaseA}, the claims follow after some computations using the explicit formula (as well as derivatives w.r.t. $ r>1 $)
		\begin{align*}
			U_0(r) = - \int_0^1\int_0^{2\pi} \dfrac{sdsd\varphi}{(r^2+s^2-2rs\cos\varphi)^{\nu/2}},
		\end{align*}
		and the fact that $ r\geq s $ in the integral. For Case \eqref{CaseB} one can use the explicit solution in Lemma \ref{lem:UnperturbedPotential}. 
	\end{proof}
	
	\section{Fr\'echet derivative of the main problem}\label{sec:Frechet}
	In this section we prove the Fr\'echet differentiability of the function $ \F $. We consider separately the stream function $ \phi_h $ and the interaction potential $ U_h\circ f_h $.
	
	\subsection{Fr\'echet derivative of the stream function}\label{linear:sec:stream}
	In this subsection we derive the Fr\'echet differential of the function $ h\mapsto \phi_h $.
	
	\begin{lem}\label{lem:StreamFunctFrechetDeriv}
		Let $ k\in \N_0 $ and $ \alpha\in (0,1) $. There exists $ \varepsilon_0>0 $ sufficiently small such that for $ B_{\varepsilon_0}\subset H_0^{k+2,\alpha} $
		\begin{align*}
			h\mapsto\phi_h \in C^1(B_{\varepsilon_0}; C^{k+2,\alpha}(\overline{\D})).
		\end{align*}
		More precisely, the linear operator $ D_h\phi_h $ is defined by $ g\mapsto D_h\phi_h[g] =:\bar{\phi} $ where 
		\begin{align}\label{eq:PotFlowDeriv}
			\begin{cases}
				\Delta \bar{\phi} = |f'_h|^2G'(\phi_h)\bar{\phi} + 2\Re\left[ (1+h')\overline{g'} \right] G(\phi_h) & \text{in } \D,
				\\
				\bar{\phi} = 0 & \text{on } \partial \D.
			\end{cases}
		\end{align}
	\end{lem}
	\begin{proof}
		First of all, the equation \eqref{eq:PotFlowDeriv} has a unique solution $ \bar{\phi} $, since $ G'\geq0 $. We apply Schauder estimates for the Laplacian and absorb the term $ |f'_h|^2G'(\phi_h)\bar{\phi} $ into the left hand side by choosing $ \norm[k+2,\alpha]{h}\leq \varepsilon_0 $ sufficiently small. This yields
		\begin{align}\label{bound1lemma3:1}
			\norm[k+2,\alpha]{\bar{\phi}}\leq C\norm[k+2,\alpha]{g},
		\end{align}
		where $ C>0 $ is independent of $ h\in B_{\varepsilon_0}\subset H_0^{k+2,\alpha} $ by Lemma \ref{lem:WellPosedStreamFunction}.
		
		Furthermore, by taking the difference of the equations for $ \phi_{h+g} $ and $ \phi_{h} $ we obtain
		\begin{align*}
			\left[ -\Delta+\int_0^1G'((1-t)\phi_h+t\phi_{h+g}) dt \right] (\phi_{h+g}-\phi_{h}) = -\left[ |f'_{h+g}|^2-|f_h'|^2 \right] G(\phi_{h+g}).
		\end{align*}
		Therefore, using Schauder estimates, we infer that
		\begin{align}\label{bound2lemma3:1}
			\norm[k+2,\alpha]{\phi_{h+g}-\phi_{h}}\leq C\norm[k+2,\alpha]{g},
		\end{align}
		where $ C>0 $ is independent of $ h $.
		
		Next we find that for $ D_h\phi_h[g]=\bar{\phi} $ and denoting $ R:=\phi_{h+g}-\phi_h-\bar{\phi} $
		\begin{align*}
			\Delta R &= \left( |f'_{h+g}|^2-|f'_{h}|^2-2\Re\left[ (1+h')\overline{g'} \right]  \right) G(\phi_{h})
			\\
			&\quad + \left( |f'_{h+g}|^2-|f'_{h}|^2 \right) G'(\phi_{h})\bar{\phi} + |f'_{h+g}|^2\left( G(\phi_{h+g})-G(\phi_{h})-G'(\phi_{h})\bar{\phi} \right) 
			\\
			&=|g'|^2G(\phi_{h}) + \left( |f'_{h+g}|^2-|f'_{h}|^2 \right) G'(\phi_{h})\bar{\phi} + G'(\phi_{h}) R 
			\\
			&\quad + \int_0^1G''((1-t)\phi_h+t\phi_{h+g}) \, dt \,  (\phi_{h+g}-\phi_{h})^2.
		\end{align*}
		Similarly as above, invoking Schauder estimates and bounds \eqref{bound1lemma3:1}-\eqref{bound2lemma3:1} we obtain (note that $ G\in C^{k+3}(\R;\R) $)
		\begin{align*}
			\norm[k+2,\alpha]{R}\leq C \norm[k+2,\alpha]{g}^2.
		\end{align*}
		Here, the constant $ C>0 $ is independent of $ h $.
		
		Finally, we need to prove that $ h\mapsto D_h\phi_h \in\mathcal{L}(H^{k+2,\alpha}_0;C^{k+2,\alpha}(\overline{\D})) $ is continuous. To this end, one has to consider differences of solutions to \eqref{eq:PotFlowDeriv} for $ h_1, h_2\in B_{\varepsilon_0} $. Applying Schauder estimates we find the bound
		\begin{align*}
			\norm[k+2,\alpha]{D_h\phi_{h_2}[g]-D_h\phi_{h_1}[g]} \leq C \norm[k+2,\alpha]{g} \left( \norm[k+2,\alpha]{h_1-h_2} + \norm[k+2,\alpha]{\phi_{h_1}-\phi_{h_2}}\right). 
		\end{align*}
		which shows the continuity property.
	\end{proof}
	
	\subsection{Fr\'echet derivative of the interaction potential}\label{linear:sec:potential}
	Here, we derive the Fr\'echet derivative of the mapping $ h\mapsto (U_h\circ f_h) (e^{i\varphi}) $. We give only the details of the proof of Case \eqref{CaseA} with $ \nu=1 $. The remaining cases can be shown in a similar way (and are in fact simpler since the integrals are less singular). We summarize the corresponding results for Case \eqref{CaseB} at the end of this section.
	
	First of all, it is convenient to apply a change of variables
	\begin{align*}
		(U_h\circ f_h) (e^{i\varphi}) = -\int_{\D} \dfrac{|f_h'(y)|^2}{|f_h(e^{i\varphi})-f_h(y)|^\nu}\, dy = -\int_{\D} \dfrac{|f_h'(e^{i\varphi}y)|^2}{|f_h(e^{i\varphi})-f_h(e^{i\varphi}y)|^\nu} \, dy.
	\end{align*}
	We then have the following result.
	\begin{pro}\label{pro:InterPotFrechetDeriv}
		Let $ U_h $ be defined as in Case \eqref{CaseA} and let $ k\in \N_0 $, $ \alpha\in (0,1) $. There exists $ \varepsilon_0>0 $ sufficiently small such that for $ h\in B_{\varepsilon_0}\subset H_0^{k+2,\alpha} $ we have that 
		\begin{align*}
			h\mapsto (U_h\circ f_h)(e^{i\varphi})\in C^1(B_{\varepsilon_0},C^{k+1,\alpha}(\T)) .
		\end{align*}
		More precisely, for $ h+g\in B_{\varepsilon_0} $ it holds that
		\begin{align*}
			D_h(U_h\circ f_h)[g](e^{i\varphi}) =& -\int_{\D}\dfrac{\sigma^1_h[g](\varphi,y)}{d_h(\varphi,y)^\nu}\, dy + \nu\int_{\D}\dfrac{\sigma^2_h[g](\varphi,y)}{d_h(\varphi,y)^{\nu+2}}\, |f_h'(e^{i\varphi}y)|^2 \, dy,
		\end{align*}
		where we define
		\begin{align}\label{eq:DefInterPotFrechetDeriv}
			\begin{split}
				d_h(\varphi,y) &:= |f_h(e^{i\varphi})-f_h(e^{i\varphi}y)|,
				\\
				\sigma^1_h[g](\varphi,y) &:= 2\Re \left[ (1+h'(e^{i\varphi}y)) \overline{g'(e^{i\varphi}y)}\right],
				\\
				\sigma^2_h[g](\varphi,y) &:= \Re \left[ \left( e^{i\varphi}(1-y) + h(e^{i\varphi})-h(e^{i\varphi}y) \right)\overline{\left( g(e^{i\varphi})-g(e^{i\varphi}y) \right)} \right].
			\end{split}
		\end{align}
	\end{pro}
	From now on we restrict ourselves to the case $ \nu=1 $. In order to prove Proposition \ref{pro:InterPotFrechetDeriv} it is convenient to introduce the following notation
	\begin{align*}
		e_h(\varphi,y)= d_h(\varphi,y)^2 = |f_h(e^{i\varphi})-f_h(e^{i\varphi}y)|^2.
	\end{align*}
	Furthermore, we need the following computation with $ t\in [0,1] $
	\begin{align}\label{eq:SecondDerivPotentialFormula}
		\begin{split}
			\dfrac{d^2}{dt^2}\left[ -\dfrac{|f'_{h+tg}(e^{i\varphi}y)|^2}{d_{h+tg}(\varphi,y)}\right] =& T^0_{h+tg,g}(\varphi,y) + T^1_{h+tg,g}(\varphi,y) +T^2_{h+tg,g}(\varphi,y),
			\\
			T^0_{h+tg,g}(\varphi,y) :=& -\dfrac{2|g'(e^{i\varphi}y)|^2}{d_{h+tg}(\varphi,y)},
			\\
			T^1_{h+tg,g}(\varphi,y) :=& \dfrac{\tau^1_{h+tg,g}(\varphi,y)}{d_{h+tg}(\varphi,y)^3},
			\\
			\tau^1_{h+tg,g}(\varphi,y) :=& 2\, \sigma^1_{h+tg}[g](\varphi,y) \, \sigma^2_{h+tg}[g](\varphi,y) + |f'_{h+tg}(e^{i\varphi}y)|^2|g(e^{i\varphi})-g(e^{i\varphi}y)|^2,
			\\
			T^2_{h+tg,g}(\varphi,y) :=& \dfrac{\tau^2_{h+tg,g}(\varphi,y)}{d_{h+tg}(\varphi,y)^5},
			\\
			\tau^2_{h+tg,g}(\varphi,y) :=& -3|f'_{h+tg}(e^{i\varphi}y)|^2 (\sigma^2_{h+tg}(\varphi,y))^2.
		\end{split}
	\end{align}
	
	\begin{lem}\label{lem:AuxilEstimateLinearization1}
		Let $ k\in \N_0 $, $ \alpha\in [0,1) $. For $ \varepsilon_0>0 $ sufficiently small and $ g,h\in  B_{\varepsilon_0}\subset H_0^{k+2,\alpha} $ the following estimates hold
		\begin{enumerate}[(i)]
			\item For $ \ell\in \N_0,\,  \ell \leq k+1 $, $ y\in \D $ we have
			\begin{align*}
				\sup_{t\in[0,1]} \left[ \sigma^2_{h+tg}[g](\cdot,y) \right]_{\ell,\alpha} \leq C\norm[k+2,\alpha]{g} |1-y|^{2-\alpha}.
			\end{align*}
			
			\item For $ \ell\in \N_0,\,  \ell \leq k+1 $, $ m=1,2 $, $ y\in \D $ we have
			\begin{align*}
				\sup_{t\in[0,1]} \left[ \tau^m_{h+tg}[g](\cdot,y) \right]_{\ell,\alpha} \leq C\norm[k+2,,\alpha]{g}^2 |1-y|^{2m-\alpha}.
			\end{align*}
			
			\item For $ \ell\in \N_0,\,  \ell \leq k+1 $, $ y\in \D $ we have
			\begin{align*}
				\sup_{t\in[0,1]} \left[ e_{h+tg}[g](\cdot,y) \right]_{\ell,\alpha} \leq C|1-y|^{2-\alpha}.
			\end{align*}
		\end{enumerate}
	\end{lem}
	\begin{proof}
		The proof of this lemma is straightforward. Indeed, one can readily check the bounds by means of the following general estimate. For any $ u\in H^{1}(\D) $ we have that
		\begin{align*}
			\dfrac{1}{|\varphi_2-\varphi_1|^\alpha} &\left| \left[ u(e^{i\varphi_1}y)-u(e^{i\varphi_1}) \right] - \left[ u(e^{i\varphi_2}y)-u(e^{i\varphi_2}) \right] \right| 
			\\
			&\leq \left( 2\norm[\infty]{u'} \right)^\alpha \left( 2|1-y| \norm[\infty]{u'} \right)^{1-\alpha} \leq 2\norm[C^1]{u} |1-y|^{1-\alpha}.
		\end{align*}
	\end{proof}
	The next lemma is also useful.
	\begin{lem}\label{lem:DistanceWithFunctLemma}
		Let $ k\in \N_0 $, $ \alpha\in [0,1) $. There is a sufficiently small $ \varepsilon_0>0 $ such that for any $ h\in B_{\varepsilon_0} \subset H_0^{k+2,\alpha} $ and $ y\in \D $ we have that
		\begin{align}\label{eq:DistanceWithFunctLemmaLowerBound}
			d_h(\varphi,y) \geq c|1-y|.
		\end{align}
		Furthermore, for any $ q\in \N $, $ n\in \N_0, \, n\leq k+1 $ and $ y\in \D $ estimate
		\begin{align*}
			\semnorm[n,\alpha]{\dfrac{1}{d_h(\cdot,y)^q}} \leq \dfrac{C_{k,q}}{|1-y|^{q+\alpha}}.
		\end{align*}
  holds. Both $ c>0 $ and $ C_{k,q>}>0 $ are independent of $ h $.
	\end{lem}
	\begin{proof}
		The first assertion follows using the mean-value theorem and choosing $ \varepsilon_0>0 $ sufficiently small. To prove the second assertion we first consider $ n=0 $. It then suffices to consider $ q=1 $. We have
		\begin{align*}
			\dfrac{1}{|\varphi_1-\varphi_2|^\alpha}\left| \dfrac{1}{d_h(\varphi_1,y)} -\dfrac{1}{d_h(\varphi_2,y)} \right| &= 	\dfrac{|d_h(\varphi_1,y)-d_h(\varphi_2,y)|}{|\varphi_1-\varphi_2|^\alpha}\dfrac{1}{d_h(\varphi_1,y)d_h(\varphi_2,y)}
			\\
			&\leq \dfrac{C}{|1-y|^2}\dfrac{|f_h(e^{i\varphi_1})-f_h(e^{i\varphi_2})+f_h(e^{i\varphi_1}y)-f_h(e^{i\varphi_2}y)|}{|\varphi_1-\varphi_2|^\alpha}
			\\
			&\leq \dfrac{C}{|1-y|^{1+\alpha}}.
		\end{align*}
		We now compute the $ n $-th order derivative, $ n\leq k+1 $. By Fa\`a di Bruno's formula, cf. Lemma~\ref{lem:FaaDiBruno}, applied to the composition of the functions $ x\mapsto x^{-q/2} $ and $ e_h=d_h^2 $ the preceding expression is a sum of terms of the form
		\begin{align*}
			\dfrac{1}{d_h(\varphi,y)^{2(\ell_1+\cdots+\ell_{n})+q}} \prod_{j=1}^{n} \left( \dfrac{d^{j}}{d\varphi^{j}}e_h(\varphi,y) \right)^{\ell_j} = \dfrac{1}{d_h(\varphi,y)^{q}}\prod_{j=1}^{n} \left( \dfrac{1}{d_h(\varphi,y)^2}\dfrac{d^{j}}{d\varphi^{j}}e_h(\varphi,y) \right)^{\ell_j},
		\end{align*}
		where $ \ell_1,\ell_2,\ldots, \ell_{n}\in \N_0 $ satisfy $  \ell_1+2 \ell_2+\cdots +n\ell_{n}=n $. The supremum norm of each term in the product $ j=1,\ldots,n $ is bounded due to Lemma \ref{lem:AuxilEstimateLinearization1} (iii) and \eqref{eq:DistanceWithFunctLemmaLowerBound}.
		
		We estimate now the seminorm $ \semnorm[\alpha]{\cdot} $ of this expression. Note that for products only one term is estimated in this seminorm while the other terms are estimated in the supremum norm. For the seminorm we apply Lemma \ref{lem:AuxilEstimateLinearization1} (iii) and the case $ n=0 $ we discussed above. Hence, the seminorm is bounded up to a constant by $ |1-y|^{-q-\alpha} $.
	\end{proof}
	
	As a result of the previous lemmas we obtain the following estimates.
	\begin{lem}\label{lem:AuxilEstimateLinearization2}
		Let $ k\in \N_0 $ and $ \alpha\in (0,1) $. We have for sufficiently small $ \varepsilon_0>0 $, $ g,h\in  B_{\varepsilon_0}\subset H_0^{k+2,\alpha} $ and $ m=0,1,2 $
		\begin{align*}
			\norm[k+1,\alpha] {\dfrac{\sigma^1_h[g](\cdot,y)}{d_h(\cdot,y)}}+\norm[k+1,\alpha] {\dfrac{\sigma^2_h[g](\cdot,y)}{d_h(\cdot,y)^3}} &\leq C\norm[k+2,\alpha]{g} |1-y|^{-1-\alpha},
			\\
			\norm[k+1,\alpha] {T^m_{h+tg,g}(\cdot,y)} &\leq C\norm[k+2,\alpha]{g}^2 |1-y|^{-1-\alpha}.
		\end{align*}
		The constant $ C>0 $ is independent of $ h,\, g $.
	\end{lem}
	\begin{proof}
		The previous lemmas can be applied without difficulty. Note that in the case of $ T^2_{h+tg,g} $ there is the factor $ d_{h+tg}^5 $ in the denominator. This is compensated by the extra factor in Lemma~\ref{lem:AuxilEstimateLinearization1} (ii) for $ m=2 $.
	\end{proof}
	
	With the previous lemmas we can give the proof of Proposition \ref{pro:InterPotFrechetDeriv}.
	\begin{proof}[Proof of Proposition \ref{pro:InterPotFrechetDeriv}]
	We consider only $ \nu=1 $. For the sake of the exposition we divide the proof into two steps.
	
 	\textit{Step 1.} We first show the Fr\'echet differentiability. We write 
		\begin{align*}
			R(\varphi) :=& (U_{h+g}\circ f_{h+g})(e^{i\varphi})-(U_h\circ f_h)(e^{i\varphi})-D_h(U_h\circ f_h)[g](e^{i\varphi})
			\\
			=& \int_{\D}\int_0^1(1-t)\dfrac{d^2}{dt^2}\left[ -\dfrac{|f'_{h+tg}(e^{i\varphi}y)|^2}{d_{h+tg}(\varphi,y)}\right] \, dt dy
			\\
			=& \int_{\D}\int_0^1(1-t)\left[T^0_{h+tg,g}(\varphi,y) + T^1_{h+tg,g}(\varphi,y) +T^2_{h+tg,g}(\varphi,y)\right] \, dt dy,
		\end{align*}
		recalling \eqref{eq:SecondDerivPotentialFormula}. We apply Lemma \ref{lem:AuxilEstimateLinearization2} to get (note that $ -1-\alpha>-2 $)
		\begin{align*}
			\norm[k+1,\alpha]{R}\leq C\norm[k+2,\alpha]{g}^2,
		\end{align*}
		which implies the Fr\'echet differentiability.
		
		\textit{Step 2.} Now, we show that $ h\mapsto D_h(U_h\circ f_h)[g] $ is continuous. First of all, one can estimate using Lemma \ref{lem:AuxilEstimateLinearization2}
		\begin{align*}
			\norm[k+1,\alpha]{D_h(U_h\circ Y_h)[g]} \leq C_k \norm[k+2,\alpha]{g}.
		\end{align*}
		where $ C_k>0 $ is independent of $ h,\, g\in B_{\varepsilon_0}\subset H_0^{k+2,\alpha} $. We can use these estimates to cut out the singularity in the integral uniformly in $ h,\, g\in B_{\varepsilon_0} $. The remaining integrand is then a smooth function with respect to $ h $. As a consequence it is continuous in $ h $. The above bounds show that this is also uniform in $ \norm[k+2,\alpha]{g} $, which ensures these estimates in the operator norm. Hence, $ h\mapsto D_h(U_h\circ f_h)[\cdot](e^{i\varphi}) \in\mathcal{L}(H^{k+2,\alpha}_0;C^{k+1,\alpha}(\T)) $ is continuous.
	\end{proof}
	Let us give the corresponding result of Proposition \ref{pro:InterPotFrechetDeriv} in Case \eqref{CaseB}. To this end, we write
	\begin{align*}
		(U_h\circ f_h) (e^{i\varphi}) = \int_{\D} |f_h'(y)|^2 \, \ln |f_h(e^{i\varphi})-f_h(y)|\, dy = \int_{\D} |f_h'(e^{i\varphi}y)|^2 \, \ln|f_h(e^{i\varphi})-f_h(e^{i\varphi}y)| \, dy.
	\end{align*}
	\begin{pro}\label{pro:InterPotFrechetDerivB}
		Let $ U_h $ be defined as in Case \eqref{CaseB} and let $ k\in \N_0 $, $ \alpha\in (0,1) $. There exists $ \varepsilon_0>0 $ sufficiently small such that for $ h\in B_{\varepsilon_0}\subset H_0^{k+2,\alpha} $ we have that
		\begin{align*}
			h\mapsto (U_h\circ f_h)(e^{i\varphi})\in C^1(B_{\varepsilon_0},C^{k+1,\alpha}(\T)).
		\end{align*}
		More precisely, for $ h+g\in B_{\varepsilon_0} $ it holds that
		\begin{align*}
			D_h(U_h\circ f_h)[g](e^{i\varphi}) =& \int_{\D} \sigma^1_h[g](\varphi,y) \, \ln d_h(\varphi,y)\, dy + \int_{\D}\dfrac{\sigma^2_h[g](\varphi,y)}{d_h(\varphi,y)^2} \, |f_h'(e^{i\varphi}y)|^2\, dy,
		\end{align*}
		where $\sigma^1_h[g], \, \sigma^2_h[g]$ and $d_h(\varphi,y)$ are given in \eqref{eq:DefInterPotFrechetDeriv}.
	\end{pro}
	
	\subsection{Fr\'echet derivative of the full problem}\label{linear:sec:fullproblem}
	Here, we give compute the Fr\'echet derivative of the second and third component of $ \F $ in \eqref{eq:FullFunction}.
	
	For the second component note that the continuous differentiability of $ (h,X)\mapsto \nabla U_h(X) $ involves no complications since $ X=(a,0) $ is assumed to be close to $ X_0=(a_0,0) $ with $ a_0\geq 2 $. Hence, $ \nabla U_h $ is smooth on a neighborhood of $ X_0 $ and
	\begin{align*}
		\nabla U_h(X) = \nu\int_{E_h}\dfrac{X-y}{|X-y|^{\nu+2}}\, dy = \nu\int_{\D}\dfrac{X-f_h(y)}{|X-f_h(y)|^{\nu+2}}|f_h'(y)|^2\, dy,
	\end{align*}
	in Case \eqref{CaseA} and
	\begin{align*}
		\nabla U_h(X) = \int_{\D}\dfrac{X-f_h(y)}{|X-f_h(y)|^2}|f_h'(y)|^2\, dy,
	\end{align*}
	in Case \eqref{CaseB}. Here, we identify $ f_h $ with a function $ \D\to \R^2 $. We obtain the following result for the Fr\'echet derivative (we again identify $ X=(a,0)\in\R^2\backsimeq\C $).
	\begin{lem}\label{lem:LinearizationSecondComponent3D}
		Let $ U_h $ be defined as in Case \eqref{CaseA} or Case \eqref{CaseB} and $ k\in \N_0 $, $ \alpha\in (0,1) $. The map $ (h,a)\mapsto \partial_{x_1} U_h(a,0) $ is continuously differentiable for $ |a-a_0|\leq \varepsilon_0 $, $ h\in B_{\varepsilon_0}\subset H_0^{k+2,\alpha} $, $ \varepsilon_0>0 $ sufficiently small, with derivative
		\begin{align*}
			D_{(h,a)} \left( \partial_{x_1} U_h(a,0) \right) [g,b] =& \partial_{x_1}^2U_h(a,0)b + W_{h,a}[g].
		\end{align*}
		The function $ W_{h,a}[g] $ is defined by
		\begin{align*}
			W_{h,a}[g]:=& \nu\int_{\D} \dfrac{-\Re[g(y)] |f_h'(y)|^2+2\Re[(1+h'(y))\overline{g'(y)}]\, \Re[a-f_h(y)] }{|a-f_h(y)|^{\nu+2}}\, dy
			\\
			& -\nu(\nu+2)\int_{\D} \dfrac{\Re[(a-f_h(y))\overline{g(y)}]\, \Re[a-f_h(y)]}{|a-f_h(y)|^{\nu+4}} |f_h'(y)|^2\, dy
		\end{align*}
		in Case \eqref{CaseA} and by
		\begin{align*}
			W_{h,a}[g]:=& \int_{\D} \dfrac{-\Re[g(y)] |f_h'(y)|^2+2\Re[(1+h'(y))\overline{g'(y)}]\, \Re[a-f_h(y)] }{|a-f_h(y)|^2}\, dy
			\\
			& -\int_{\D} \dfrac{2\Re[(a-f_h(y))\overline{g(y)}]\, \Re[a-f_h(y)]}{|a-f_h(y)|^4} |f_h'(y)|^2\, dy
		\end{align*}
		in Case \eqref{CaseB}, respectively.
	\end{lem}
	
	 The mass constrain, i.e. the third component of $ \F $, leads to the following Fr\'echet derivative.
	\begin{lem}\label{lem:LinearizationMass}
		The Fr\'echet derivative of the map
		\begin{align*}
			h\mapsto |E_h| = \int_{\D} |f_h'(x)|^2 \, dx
		\end{align*} 
		is given by
		\begin{align*}
			g\mapsto 2\Re \int_{\D} (1+h')\overline{g'}\, dx.
		\end{align*}
	\end{lem}
	We omitted the proof of the previous Lemmas since they can be easily checked. Finally, the following differentiability result follows from combining Lemma \ref{lem:StreamFunctFrechetDeriv}, Proposition \ref{pro:InterPotFrechetDeriv}, respectively Proposition \ref{pro:InterPotFrechetDerivB}, and Lemmas \ref{lem:LinearizationSecondComponent3D}, \ref{lem:LinearizationMass}.
	\begin{pro}\label{pro:MainProbFrechetDeriv}
		Let $ k\in \N_0 $, $ \alpha\in (0,1) $. There is $ \varepsilon_0>0 $ sufficiently small such that $ \F\in C^1(U;\mathbb{Z}^{k+1,\alpha})  $, where
		\begin{align*}
			U=B_{\varepsilon_0}(0)\times (a_0-\varepsilon_0,a_0+\varepsilon_0)\times (\lambda_0-\varepsilon_0,\lambda_0+\varepsilon_0) \subset \mathbb{X}^{k+2,\alpha}.
		\end{align*}
	\end{pro}
	
	\section{Invertibility of the linearized operator}\label{sec:inver}
	In order to apply the implicit function theorem we need to invert the linearized operator at the point $ (0,X_0,\lambda_0,0) $, i.e. the linear operator
	\begin{align}\label{eq:FrechetDerivativAtZero}
		\mathbb{X}^{k+2,\alpha}\to \mathbb{Z}^{k+1,\alpha}: \, 	(g,b,\mu)\mapsto D_{(h,a,\lambda)}\F(0,a_0,\lambda_0,0)[g,b,\mu].
	\end{align}
	We recall that the functional spaces are defined in \eqref{eq:DefSpaces}. It is convenient to write the function $ g\in H_0^{k+2,\alpha} $ using power series of the form
	\begin{align}\label{eq:PowerSeriesExp}
		g(z) = \sum_{n\geq0}\hat{g}_nz^{n+1}.
	\end{align}
	Recall that $ g(0)=0 $ and $ g'(0)=\hat{g}_0\in \R $ since $ g\in H^{k+2,\alpha}_0 $.
	\begin{rem}\label{rem:ConformalMapPowerSeries}
		Let us briefly comment on the form of the power series \eqref{eq:PowerSeriesExp}. 
		\begin{enumerate}[(i)]
			\item We choose an index shift in the coefficients of \eqref{eq:PowerSeriesExp}, in order that the linearized operator \eqref{eq:FrechetDerivativAtZero} is diagonalized when using a Fourier decomposition, cf. Lemma \ref{lem:LinearizationZero}.
			\item With this choice the coefficient $ \hat{g}_0 $ corresponds to a rescaling $ z\mapsto (1+\hat{g}_0)z $. Consequently, it appears in the linearization of $ h\mapsto |E_h| $, cf. Lemma \ref{lem:LinearizationZero}. In the Fourier series it appears as the zeroth coefficient.
			\item Furthermore, infinitesimal translations are given by the conformal mappings $ T_\varepsilon: z\mapsto z+\varepsilon $ for small $ \varepsilon>0 $. In order to satisfy the conditions $ T_\varepsilon(0)=0 $ and $ T'_\varepsilon(0)\in \R $ we use a Blaschke factor, see \eqref{eq:Blaschke}, yielding the conformal mapping
			\begin{align}\label{eq:ConformalMapTranslation}
				z\mapsto \dfrac{z-\varepsilon}{1-\varepsilon z}+\varepsilon = \dfrac{z-\varepsilon^2z}{1-\varepsilon z} = z + h_\varepsilon(z), \quad h_\varepsilon(z) = \dfrac{\varepsilon z^2-\varepsilon^2 z}{1-\varepsilon z} = \varepsilon z^2 + \mathcal{O}(\varepsilon^2)
			\end{align}
			as $\varepsilon\to0$. In particular, infinitesimal translations correspond to the coefficient of $ z^2 $, i.e. $ \hat{g}_1 $ in \eqref{eq:PowerSeriesExp}. In Fourier series they correspond to the coefficients for $ e^{\pm i\varphi} $, which is $ \hat{g}_1 $ respectively $ \overline{\hat{g}_1} $ in the linearization, cf. Lemma \ref{lem:LinearizationZero}.
		\end{enumerate}
	\end{rem}
	The main result of this section is the following proposition.
	\begin{pro}\label{pro:InvertLinOp}
		Let $ k\in \N_0 $, $ \alpha\in(0,1) $. The operator \eqref{eq:FrechetDerivativAtZero} is an isomorphism under the assumptions of Theorem \ref{thm:MainThm}.
	\end{pro}
	For the purpose of proving Proposition \ref{pro:InvertLinOp} it is necessary to compute explicitly the form of the linear operator \eqref{eq:FrechetDerivativAtZero}. This is done in the following subsections, which also contain further needed auxiliary results.
	
	\subsubsection*{Linearization of the stream function.}\label{subsec:LinearizationStreamFunction}
	For the proof we write the operator $ \bar{\phi}[g]:=D_h\phi_h(0) [g] $, i.e. the Fr\'echet derivative of $ \phi_h $ at $ h=0 $, more explicitly. Due to Lemma \ref{lem:StreamFunctFrechetDeriv} it solves the equation
	\begin{align}\label{eq:LinEqStreamFunct}
		\Delta \bar{\phi} = G'(\phi_0)\bar{\phi} + 2\Re\left[ g' \right] G(\phi_0), \quad \bar{\phi}\mid_{\partial\D}=0.
	\end{align}
	Hence, using the expression \eqref{eq:PowerSeriesExp} we find that
	\begin{align}\label{eq:DefAuxilFunct}
		2\Re\left[ g'(re^{i\varphi}) \right] = 2\Re\left[ \sum_{n\geq0}(n+1)\hat{g}_nr^{n}e^{in\varphi} \right] = \sum_{n\in\Z}(|n|+1)\hat{\xi}_n[g] \, r^{|n|}e^{in\varphi},
	\end{align}
	where the coefficients $\hat{\xi}_{n}[g]$ are given by
	\begin{align}\label{coef:gn}
		\hat{\xi}_n[g]:=
		\begin{cases}
			\hat{g}_n & n\geq1,
			\\
			2\hat{g}_0 & n=0,
			\\
			\overline{\hat{g}_n} & n\leq-1.
		\end{cases}
	\end{align}
	Recall that $ g'(0)=\hat{g}_0\in \R $ since $ g\in H^{k+2,\alpha}_0 $. We use a Fourier decomposition to obtain the formula
	\begin{align*}
		\bar{\phi}[h](r,\varphi) = \sum_{n\in \Z} (|n|+1)\hat{\xi}_n[g] \,  A_n(r) e^{in \varphi},
	\end{align*}
	where $ A_n(r) $ solves the ordinary differential equation, see \eqref{eq:LinearizationStreamFunctODE},
	\begin{align}\label{ode:eq:linearizationstream2}
		\dfrac{1}{r}(rA_n')' -\dfrac{n^2}{r^2}A_n-G_1A_n = r^{|n|}G_0, \quad A_n(1)=0.
	\end{align}
	Above we used the shortcut notation $ G_0(r):=G(\phi_0(r))$, $ G_1(r):=G'(\phi_0(r)) $. Moreover, notice that $ A_n=A_{-n}$  by symmetry. 
	
	The function $ \bar{\phi} $ enters the linearization of $ \F $ in the following way
	\begin{align}\label{eq:ExplicitLinearizedStreamFunction}
		\left[\nabla\phi_0\cdot \nabla\bar{\phi}\right](e^{i\varphi}) = \phi_0'(1)\partial_r\bar{\phi}(1,\varphi) = \phi_0'(1)\sum_{n\in \Z} (|n|+1)\hat{\xi}_n[g] A_n'(1) e^{in \varphi}.
	\end{align}
	Recall that the unperturbed stream function $ \phi_0 $ is radial. Hereafter we provide a crucial result concerning the asymptotics of $ A_n'(1) $ as $ n\to \infty $.
	\begin{lem}\label{lem:AsymptotcisCoeff}
		Consider the solution $ \bar{\phi} $ of \eqref{eq:LinEqStreamFunct}. Then, the coefficients $ A_n $ have the asymptotics
		\begin{align}\label{asymp:formu}
			\lim_{n\to \infty} n A_n'(1) = \dfrac{G(\phi_0(1))}{2}.
		\end{align}
	\end{lem}
	\begin{rem}
		Compare \eqref{asymp:formu} with the explicit solutions for $ G\equiv -2\Omega_0 $ in Remark \ref{rem:NonResonanceCondition} (iii).
	\end{rem}
	\begin{proof}[Proof of Lemma \ref{lem:AsymptotcisCoeff}]
		Let $ n\geq1 $ throughout the proof. Writing $ \widetilde{\phi}_n(r,\varphi)=A_n(r)e^{in\varphi} $ we have
		\begin{align} 
			(\Delta-G_1)\widetilde{\phi}_n = r^ne^{in\varphi}G_0, \quad \widetilde{\phi}_n(1,\varphi)=0.
		\end{align}
		Since $ G_1= G'\circ \phi_0\geq  \textcolor{magenta}{0}$, the operator $ \Delta-G_1 $ with zero boundary conditions is invertible. Furthermore, since $ G_0\in C^{k+3,\alpha} $ we have $ \widetilde{\phi}_n\in C^{k+5,\alpha}(\overline{\D}) $.
		
		Let us look at the following auxiliary ODE
		\begin{align*}
			\dfrac{1}{r}(rb_n')' -\dfrac{n^2}{r^2}b_n = r^{n}G_0, \quad b_n(1)=0.
		\end{align*}
		Note that comparing it with the ODE solved by $ A_n $ given in \eqref{ode:eq:linearizationstream2}, only the term $ G_1 $ is removed, which is expected to be of lower order for $ n\to \infty $. It is convenient to write $ b_n(r)=r^n\beta_n(r) $ with
		\begin{align*}
			\dfrac{1}{r}(r\beta_n')' +\dfrac{2n}{r}\beta_n' = G_0, \quad \beta_n(1)=0.
		\end{align*}
		We can find the solution explicitly up to a parameter
		\begin{align*}
			\beta_n(r)= -\dfrac{\beta'_n(1)}{2n} \dfrac{1-r^{2n}}{r^{2n}}-\dfrac{1}{2n}\int_r^1G_0(s)s \, ds + \dfrac{1}{2n}\dfrac{1}{r^{2n}}\int_r^1 G_0(s) s^{2n+1}\, ds.
		\end{align*} 
		In order that $ \beta_n $ exists for $ r\to 0 $ we choose 
		\begin{align}\label{eq:ProofAsymptoticsDeriv}
			B_n:=\beta'_n(1) :=\int_0^1 G_0(s) s^{2n+1} \, ds
		\end{align}
		yielding
		\begin{align*}
			\beta_n(r)= \dfrac{B_n}{2n} -\dfrac{1}{2n}\int_r^1G_0(s)s \, ds - \dfrac{1}{2n}\dfrac{1}{r^{2n}}\int_0^r G_0(s) s^{2n+1} \, ds.
		\end{align*}
		Observe that $ |\beta_n(r)|\leq C/n $ for some constant $ C>0 $ independent of $ r $ and $ n $.
		
		Let us now decompose
		\begin{align*}
			\widetilde{\phi}_n(r,\varphi) = \widetilde{\phi}_n^1(r,\varphi)+\widetilde{\phi}_n^2(r,\varphi), \quad \widetilde{\phi}_n^1(r,\varphi) := r^n\beta_n(r) e^{in\varphi}.
		\end{align*}
		Accordingly, we get the decomposition
		\begin{align}\label{eq:ProofAsymptoticsDecomp}
			A_n(r) = A_n^1(r)+A_n^2(r), \quad A_n^1(r)=r^n\beta_n(r).
		\end{align}
		Hence, with the above calculations we have
		\begin{align*}
			(\Delta-G_1)\widetilde{\phi}_n^2 = (\Delta-G_1)(\widetilde{\phi}_n-\widetilde{\phi}_n^1) = G_1\widetilde{\phi}_n^1, \quad \widetilde{\phi}_n^2\mid_{\partial\D} =0.
		\end{align*}
		We can apply regularity estimates in Sobolev spaces to obtain
		\begin{align*}
			\norm[W^{2,2}(\D)]{\widetilde{\phi}_n^2}\leq C\norm[L^2(\D)]{g\widetilde{\phi}_n^1} \leq \dfrac{C}{n} \left( \int_0^1s^{2n+1} \, ds \right)^{1/2} \leq \dfrac{C}{n^{3/2}}.
		\end{align*}
		Here, we used $ |\widetilde{\phi}_n^1(r,\varphi)|\leq |\beta_n(r)|\leq C/n $. Applying the trace theorem (cf. \cite{GilbargTrudinger2001}) gives
		\begin{align*}
			\norm[L^2(\partial\D)]{\partial_r\widetilde{\phi}_n^2(1,\cdot)}\leq \dfrac{C}{n^{3/2}},
		\end{align*}
		and hence $ |(A_n^2)'(1)|\leq C/n^{3/2} $. In conclusion, combining \eqref{eq:ProofAsymptoticsDecomp} and \eqref{eq:ProofAsymptoticsDeriv} we find that
		\begin{align*}
			\lim_{n\to \infty} n\, A_n'(1) =& \lim_{n\to \infty} n\, (A_n^1)'(1) = \lim_{n\to \infty} n \, \beta_n'(1) = \lim_{n\to \infty} n \, B_n \\
			=& \lim_{n\to \infty} n\int_0^1G_0(s) s^{2n+1}\, ds 
			\\
			=& \lim_{n\to \infty} n\left( \dfrac{G_0(1)}{2n+2} - \dfrac{1}{2n+2}\int_0^1G_0'(s) s^{2n+2}\, ds \right) 
			\\
			=& \dfrac{G_0(1)}{2} = \dfrac{G(\phi_0(1))}{2},
		\end{align*}
		showing the desired result.
	\end{proof}
	
	\subsubsection*{Linearization of the interaction potential, Case \eqref{CaseA}.}\label{subsec:LinearizationPotCaseA}
	Due to Proposition \ref{pro:InterPotFrechetDeriv} we have for $ x=e^{i\varphi} $ (we use here the change of variables $ e^{i\varphi}y\mapsto y $)
	\begin{align*}
		D_h(U_h\circ f_h)\mid_{h=0}[g](x)= -\int_{\D} \dfrac{2\Re[g'(y)]}{|x-y|^\nu}\, dy + \nu\int_{\D} \dfrac{\Re[(\overline{x-y})(g(x)-g(y))]}{|x-y|^{2+\nu}}\, dy.
	\end{align*}
	As before we use the power series expansion for $ g $ given in \eqref{eq:PowerSeriesExp}. We have
	\begin{align*}
		D_h(U_h\circ f_h)\mid_{h=0}[g](x)= \sum_{n=0}^{\infty}\Re \left[ \hat{g}_{n}\int_{\D} \left( \nu\dfrac{x^{n+1}-y^{n+1}}{x-y}-2(n+1)y^{n} \right) \dfrac{dy}{|x-y|^\nu} \right] .
	\end{align*}
	Since $ x=e^{i\varphi} $ we can use a rotation to obtain
	\begin{align*}
		\sum_{n=0}^{\infty}\Re \left[ \hat{g}_{n}e^{in\varphi}\int_{\D} \left( \nu \dfrac{1-y^{n+1}}{1-y}-2(n+1)y^{n} \right) \dfrac{dy}{|1-y|^\nu} \right].
	\end{align*}
	Recalling the definition of $ c_n $ in \eqref{eq:GravPotLinCoeffA}, the fact that $c_{n}$ are real and the definition of $ \hat{\xi}_n[g] $ in \eqref{eq:DefAuxilFunct}, we obtain
	\begin{align}\label{eq:ExplicitLinearizedGravPot}
		D_h(U_h\circ f_h)\mid_{h=0}[g](e^{i\varphi}) = \sum_{n\in \Z} c_n \, \hat{\xi}_n[g] \, e^{in\varphi},
	\end{align}
	where we define $ c_{n}=c_{|n|} $ for $ n<0 $. The following result will be useful.
	\begin{lem}\label{lem:AsymptotcisCoeffGravPot}
		The sequence $ c_n $ defined in \eqref{eq:GravPotLinCoeffA} with $ \nu=1 $ satisfies
		\begin{align}\label{eq:AsymptoticsCoeffA}
			\lim_{n\to \infty} \dfrac{c_n}{\ln n}=\gamma_0
			\quad \text{with} \quad 
			\gamma_0:= \int_0^\infty\int_0^\infty \dfrac{e^{-r}\zeta \sin(\zeta)}{(r^2+\zeta^2)^{3/2}}\, drd\zeta.
		\end{align}
		For $ \nu\in (0,1) $ we have $ \sup_{n\geq0} |c_n|<\infty $.
	\end{lem}
	\begin{proof}
		First of all, we have
		\begin{align*}
			c_n=\sum_{k = 0}^{n}\dfrac{\nu}{2}\int_{\D}\dfrac{y^k}{|1-y|^\nu}\, dy-(n+1) \int_{\D}\dfrac{y^{n}}{|1-y|^\nu}\, dy.
		\end{align*}
		Let us define
		\begin{align*}
			\widetilde{c}_k:=\dfrac{1}{2}\int_{\D} \dfrac{y^k}{|1-y|^\nu}\, dy.
		\end{align*}
	
		We show below that 
		\begin{align}\label{eq:ProofAsymptoticsGrav}
			\lim_{k\to \infty} k^{2-\nu} \, \widetilde{c}_k = \gamma_0^\nu, \quad \gamma_0^\nu := \nu\int_0^\infty\int_0^\infty \dfrac{e^{-r}\zeta \sin(\zeta)}{(r^2+\zeta^2)^{(2+\nu)/2}}\, drd\zeta.
		\end{align}
		Note that $ \gamma_0^1=\gamma_0 $ for $ \nu=1 $. With this we infer for $ \nu=1 $
		\begin{align*}
			\lim_{n\to \infty} \dfrac{c_n}{H_n}=\gamma_0, \quad H_n= \sum_{k = 1}^{n} \dfrac{1}{k}.
		\end{align*}
		Since $ H_n= \ln(n) \, (1+o(1)) $ as $ n\to \infty $, this implies the asymptotics \eqref{eq:AsymptoticsCoeffA} for $ \nu=1 $. The claim for $ \nu\in (0,1) $ is also a consequence of the above asymptotics. 
		
		We now prove \eqref{eq:ProofAsymptoticsGrav}. The term $ \widetilde{c}_k $ is real-valued so that
		\begin{align}\label{eq:ProofAsymptoticsGravSplitting}
			\widetilde{c}_k = \dfrac{1}{2}\int_0^1\int_0^{2\pi}\dfrac{r^{k+1}\cos(k\varphi)}{(1+r^2-2r\cos\varphi)^{\nu/2}}\, d\varphi dr = I_k^1+I_k^2.
		\end{align}
		The two terms $ I_k^1 $ and $ I_k^2 $ are defined by splitting the integral \eqref{eq:ProofAsymptoticsGravSplitting} with respect to $ r $ into the regions $ (0,1/2) $ and $ (1/2,1) $. We can readily check that 
		\begin{align}\label{eq:ProofAsymptoticsGravII}
			k^{2-\nu} |I_k^1| \leq \dfrac{Ck^{2-\nu}}{2^{k+1}}\to 0.
		\end{align}
		To deal with the $ I_{k}^2 $ we notice that with the change of variables $ r=1-s $
		\begin{align*}
			k^{2-\nu} \, I_k^2 &= \dfrac{k^{2-\nu}}{2} \int_0^{1/2}\int_0^{2\pi} \dfrac{(1-s)^{k+1}\cos(k\varphi)}{\left(s^2+4(1-s)\sin^2(\varphi/2) \right)^{\nu/2}}\, d\varphi ds
			\\
			&= \dfrac{k^{1-\nu}}{2} \int_0^{k/2}\int_0^{2\pi} \dfrac{\left( 1-\dfrac{r}{k} \right)^{k+1}\cos(k\varphi)}{ \left[  \left(\dfrac{r}{k} \right)^2+4\left( 1-\dfrac{r}{k} \right)\sin^2\left(\dfrac{\varphi}{2}\right) \right]^{\nu/2} }\, d\varphi dr.
		\end{align*}
		In the second equality we used the change of variables $ ks=r $. Furthermore, writing $ k\varphi=\psi $ we get
		\begin{align*}
			k^{2-\nu} \, I_k^2 &= \dfrac{1}{2} \int_0^{k/2}\int_0^{2k\pi} \dfrac{\left( 1-\dfrac{r}{k} \right)^{k+1}\cos(\psi)}{ \left[  r^2+4\left( 1-\dfrac{r}{k} \right)k^2\sin^2\left(\dfrac{\psi}{2k}\right) \right]^{\nu/2}  }\, d\psi dr
			\\
			&=\int_0^{k/2}\int_0^{k\pi} \dfrac{\left( 1-\dfrac{r}{k} \right)^{k+1}\cos(\psi)}{ \left[  r^2+4\left( 1-\dfrac{r}{k} \right)k^2\sin^2\left(\dfrac{\psi}{2k}\right) \right]^{\nu/2} }\, d\psi dr,
		\end{align*}
		where we used the symmetry in the last equality. Let us now define the function $ \zeta_k:(0,k\pi)\to (0,2k) $ by
		\begin{align*}
			\zeta_k(\psi)=2k\sin\left(\dfrac{\psi}{2k}\right),
		\end{align*}
		which is one-to-one and onto. Furthermore, by a Taylor expansion one can see that $ \zeta_k(\psi)\to \psi $ for any $ \psi\in (0,k\pi) $ as $ k\to \infty $. Consequently, we have for the inverse function $ \psi_k(\zeta)\to \zeta $ as $ k\to \infty $. We obtain by the change of variables $ \psi\mapsto \zeta $
		\begin{align*}
			k^{2-\nu} \, I_k^2 = \int_0^{k/2}\int_0^{2k} \dfrac{\left( 1-\dfrac{r}{k} \right)^{k+1}}{ \left[ r^2+\left( 1-\dfrac{r}{k} \right)\zeta^2 \right]^{\nu/2}  } \cos(\psi_k(\zeta))\psi_k'(\zeta)d\zeta dr.
		\end{align*}
		We now use an integration by parts in $ \zeta $ to obtain (note that the boundary terms vanish, due to $ \psi_k(0)=0,\psi_k(2k)=k\pi $)
		\begin{align*}
			k^{2-\nu} \, I_k^2 = \nu\int_0^{k/2}\int_0^{2k} \dfrac{\left( 1-\dfrac{r}{k} \right)^{k+2}\zeta\sin(\psi_k(\zeta))}{\left[ r^2+\left( 1-\dfrac{r}{k} \right)\zeta^2 \right]^{(\nu+2)/2} }d\zeta dr.
		\end{align*}
		The integrand converges pointwise to 
		\begin{align*}
			\dfrac{e^{-r}\zeta \sin(\zeta)}{(r^2+\zeta^2)^{(\nu+2)/2}}
		\end{align*}
		as $ k\to \infty $. Since 
		\begin{align*}
			\left( 1-\dfrac{r}{k} \right)^{k+2} &= \exp\left( (k+2)\ln\left(1-\dfrac{r}{k}\right) \right) \leq e^{-r},
			\\
			\psi_k(\zeta) &= 2k\arcsin\left(  \dfrac{\zeta}{2k}\right) \leq C\zeta
		\end{align*}
		for say $ \zeta\in(0,1) $, a majorant is given by
		\begin{align*}
			\dfrac{e^{-r}\min(C\zeta^2,\zeta)}{(r^2+\zeta^2/2)^{(2+\nu)/2}}.
		\end{align*}
		Hence, we get $ k^{2-\nu} I_k^2\to \gamma_0^\nu $. Combining this with \eqref{eq:ProofAsymptoticsGravII} and \eqref{eq:ProofAsymptoticsGravSplitting} yields \eqref{eq:ProofAsymptoticsGrav}.
	\end{proof}
	
	\subsubsection*{Linearization of the interaction potential, Case \eqref{CaseB}.}\label{subsec:LinearizationPotCaseB}
	By Proposition \ref{pro:InterPotFrechetDerivB} we have for $ x=e^{i\varphi} $
	\begin{align*}
		D_h(U_h\circ f_h)\mid_{h=0}[g](x)= 2\int_{\D} \ln|x-y| \, \Re[g'(y)]\, dy+ \int_{\D} \dfrac{\Re[(\overline{x-y})(g(x)-g(y))]}{|x-y|^2}\, dy.
	\end{align*} 
	Again, we use the power series expansion for $ g $, cf. \eqref{eq:PowerSeriesExp}, yielding
	\begin{align*}
		D_h(U_h\circ f_h)\mid_{h=0}[g]\left( e^{i\varphi} \right) = \sum_{n=0}^{\infty}2\Re \left[ \hat{g}_{n}\int_{\D} \left( (n+1)y^n\ln|x-y|+\dfrac{1}{2}\dfrac{x^{n+1}-y^{n+1}}{x-y} \right) \, dy \right] .
	\end{align*}
	For $ x=e^{i\varphi} $ and applying the change of variables $ y\mapsto e^{i\varphi}y $ gives
	\begin{align*}
		\sum_{n=0}^{\infty}2\Re \left[ \hat{g}_{n} e^{in\varphi}\int_{\D} \left( (n+1)y^n\ln|1-y|+\dfrac{1}{2}\dfrac{1-y^{n+1}}{1-y} \right) \, dy \right] .
	\end{align*}
	As we will see below we have, see \eqref{eq:GravPotLinCoeffB},
	\begin{align}\label{eq:ComputationGravPotLinCoeffB}
		c_n=\int_{\D} \left( (n+1)y^n\ln|1-y|+\dfrac{1}{2}\dfrac{1-y^{n+1}}{1-y} \right) \, dy = \begin{cases}
			\frac{\pi}{2}\left(1-\frac{1}{n}\right) & n\geq 1,
			\\
			\frac{\pi}{2} & n=0.
		\end{cases}
	\end{align}
	Recalling the definition of $ \hat{\xi}_n[g] $ in \eqref{eq:DefAuxilFunct}, we have (see \eqref{eq:ExplicitLinearizedGravPot})
	\begin{align*}
		D_h(U_h\circ f_h)\mid_{h=0}[g](e^{i\varphi}) = \sum_{n\in \Z} c_n\, \hat{\xi}_n[g] \, e^{in\varphi},
	\end{align*}
	where we again define $ c_{n}=c_{|n|} $ for $ n<0 $.
	
	Let us now prove \eqref{eq:ComputationGravPotLinCoeffB}. For $ n=0 $ the integral reduces to $ U_0(1)+\pi/2=\pi/2 $, cf. Lemma \ref{lem:UnperturbedPotential}. For the other cases let us first observe that
	\begin{align}\label{eq:ProofComputationGravPotLinCoeffB}
		\dfrac{1}{2}\int_{\D}\dfrac{1-y^{n+1}}{1-y} \, dy = \dfrac{1}{2} \sum_{k = 0}^n \int_{\D} y^k\, dy = \dfrac{1}{2} \sum_{k = 0}^n \int_0^{2\pi}\int_0^1 r^ke^{ik\varphi}\, rdrd\varphi = \dfrac{\pi}{2}. 
	\end{align}
	Moreover, we can also write
	\begin{align}\label{eq:ProofComputationGravPotLinCoeffB2}
		(n+1)\int_{\D} y^n\ln|1-y| \, dy = (n+1) \int_0^1\int_0^{2\pi}r^ne^{in\varphi}\ln|1-re^{i\varphi}|\, rdrd\varphi.
	\end{align}
	Hence using the following expansion
	\begin{align*}
		\ln|1-re^{i\varphi}| &= \dfrac{1}{2}\left( \ln(1-re^{i\varphi})+\ln(1-re^{-i\varphi})  \right) 
		\\
		&= -\dfrac{1}{2}\left( \sum_{k = 0}^\infty \dfrac{r^{k+1}e^{i(k+1)\varphi}}{k+1}+\sum_{k = 0}^\infty \dfrac{r^{k+1}e^{-i(k+1)\varphi}}{k+1} \right).
	\end{align*}
	and plugging it in \eqref{eq:ProofComputationGravPotLinCoeffB2} we find that
	\begin{align*}
		(n+1)\int_{\D} y^n\ln|1-y| \, dy = -2\pi\dfrac{n+1}{2n}\int_0^1 r^{2n+1}\, dr = -\dfrac{\pi}{2n}.
	\end{align*}	
	Thus, combining \eqref{eq:ProofComputationGravPotLinCoeffB} and \eqref{eq:ProofComputationGravPotLinCoeffB2} we infer \eqref{eq:ComputationGravPotLinCoeffB}.
	
	\begin{rem}
		Let us note that in both Case \eqref{CaseA} and Case \eqref{CaseB} we have $ c_1=0 $. This holds in general since $ (U_{h_\varepsilon}\circ f_{h_\varepsilon})(1)=U_0(1) $, where $ h_\varepsilon(z) = \varepsilon z^2 +\mathcal{O}(\varepsilon^2) $ is associated to translations, see formula \eqref{eq:ConformalMapTranslation} in Remark \ref{rem:ConformalMapPowerSeries}. We hence obtain
		\begin{align*}
			c_1 = D_h(U_h\circ f_h)\mid_{h=0}[z^2](1) =  \dfrac{d}{d\varepsilon} \mid_{\varepsilon=0} (U_{h_\varepsilon}\circ f_{h_\varepsilon})(1) = 0.
		\end{align*}
		As mentioned in Remark \ref{rem:ConformalMapPowerSeries} the effects of conformal mappings due to translations appear to first order in the Fourier modes $ n=\pm 1 $ and thus in the coefficient $ c_1 $. 
	\end{rem}
	
	\subsubsection*{Linearization of the full problem.}\label{subsec:LinearizationFullProb}
	We summarize the full linearized operator at $ (h_0\equiv0,a_0,\lambda_0,m=0) $ in the following lemma.
	\begin{lem}\label{lem:LinearizationZero}
		The operator $ D_{(h,a,\lambda)}\F(0,a_0,\lambda_0,0) $ has the form
		\begin{align*}
			(g,b,\mu)&\mapsto \left( \begin{matrix}
				\Lin g -\mu
				\\
				\Omega_0^2b-\partial_{x_1}^2U_0(a_0,0)b - W_{0,a_0}[g]
				\\
				\pi\hat{h}_0
			\end{matrix}\right) ,
			\\ 
			\Lin g(\varphi) &:= 2\omega_0\hat{g}_0 + \sum_{n\geq1} \omega_n\, \hat{g}_n e^{in\varphi}+\sum_{n\leq -1} \omega_{n}\, \overline{\hat{g}_{|n|}} e^{in\varphi},
			\\
			\omega_n&=-\dfrac{1}{2}\phi_0'(1)^2(|n|+1)+\phi_0'(1)A_n'(1) (|n|+1)-\dfrac{1}{2}\Omega_0^2+c_{|n|},
		\end{align*}
		Here, $ W_{0,a_0}[g] $ is defined in Lemma \ref{lem:LinearizationSecondComponent3D} in both Case \eqref{CaseA} and Case \eqref{CaseB}.
	\end{lem}
	Note that in the last component of the linearized operator we identify again $ \R^2\backsimeq\C $. Furthermore, the coefficients $ \omega_n $ appeared already in \eqref{eq:FourierMultipliers}.
	\begin{proof}[Proof of Lemma \ref{lem:LinearizationZero}]
		The first component of $ \F $ in \eqref{eq:FullFunction} has the linearization at the point $ (h=0,X_0,\lambda_0,m=0) $
		\begin{align*}
			(g,\mu)\mapsto &\phi_0'(1) \, \partial_r\bar{\phi}(1,\varphi)- \phi_0'(1)^2 \Re[g'(e^{i\varphi})] 
			\\
			&- \Omega_0^2 \, \Re\left[ e^{-i\varphi}g(e^{i\varphi}) \right] +D_h(U_h\circ f_h)\mid_{h=0}[g](e^{i\varphi})- \mu.
		\end{align*}
		Using \eqref{eq:PowerSeriesExp} and \eqref{eq:DefAuxilFunct} we obtain
		\begin{align*}
			\Re\left[ e^{-i\varphi}g(e^{i\varphi}) \right] = \dfrac{1}{2} \sum_{n\in\Z} \hat{\xi}_n[g]\, e^{in\varphi}.
		\end{align*}
		Using both \eqref{eq:ExplicitLinearizedStreamFunction} and \eqref{eq:ExplicitLinearizedGravPot} yields the expression of the first component. Applying the definition of $ \hat{\xi}_n[g] $ in \eqref{eq:DefAuxilFunct} yields the form of the operator $ \Lin g $.
		
		The linearization of the second component $ \F $ in \eqref{eq:FullFunction} is a consequence of Lemma \ref{lem:LinearizationSecondComponent3D}. For the last component, note that the linearization of the mass constraint in Lemma \ref{lem:LinearizationMass} becomes $ g\mapsto \pi\, \Re[\hat{g}_0]= \pi \hat{g}_0 $, since $ g\in H_0^{k+2,\alpha} $. This concludes the proof.
	\end{proof}
	Before providing the proof of Proposition  \ref{pro:InvertLinOp} we need to show the following result on Fourier multipliers on the torus in H\"older spaces.
	\begin{lem}\label{lem:PeriodicPseudoDiffReg}
		Let $ k\in\N $ and $ \alpha\in(0,1) $. Consider a sequence $ \beta = (\beta_n)_n $ of the form $ \beta_n=\kappa/(|n|+b_n) $, $ \beta_0=0 $, $ n\in\Z $ with some real constant $ \kappa\neq0 $. Assume that $ b_n\neq -|n| $ is a sequence satisfying $ \sup_{n\neq0}|b_n||n|^{-\gamma}\leq C $ for some $ 0\leq \gamma\leq1/2 $. Then, the periodic pseudodifferential operator
		\begin{align*}
			\textsf{OP}(\beta)\xi (\varphi) = \sum_{n\in \Z\backslash\{0\}} \beta_n \hat{\xi}_n\, e^{in\varphi}
		\end{align*}
		defines a bounded map $ C_0^{k,\alpha}(\T)\to C_0^{k+1,\alpha}(\T) $.
	\end{lem}
	\begin{proof}
		Recall that the Hilbert transform $ \mathscr{H} $ defined by the Fourier multipliers $ -i\sgn(n) $ is a bounded map $ C_0^{k,\alpha}(\T)\to C_0^{k,\alpha}(\T) $ for all $ k\in \N $, $ \alpha\in (0,1) $. Since the operator with multiplier $ 1/in $ corresponds to integration, we conclude that the operator with multiplier $ 1/|n| = i\sgn(n) / in $ is a bounded map $ C_0^{k,\alpha}(\T)\to C_0^{k+1,\alpha}(\T) $.
		
		We now write
		\begin{align*}
			\beta_n = \dfrac{\kappa}{|n|} - \dfrac{\kappa}{|n|}\cdot \dfrac{b_n}{(|n|+b_n)} = \dfrac{\kappa}{|n|}\left( 1+r_n \right) .
		\end{align*}
		By assumption it holds $ c_1\leq |1+b_n/|n|| $ for some constant $ c_1>0 $. Hence, we have
		\begin{align*}
			|r_n| \leq \dfrac{|b_n|}{c_1|n|}\leq \dfrac{C}{|n|^{1-\gamma}}\leq \dfrac{C}{|n|^{1/2}}.
		\end{align*} 
		Thus, the sequence $ r=(r_n)_n $ satisfies the $ \rho $-condition in \cite[Theorem 3.1]{Cardona2017HolderPeriodicPseudoDiff} with $ \rho=1/2 $ and hence $ \textsf{OP}(r) $ constitutes a bounded map $ C_0^{k,\alpha}(\T)\to C_0^{k,\alpha}(\T) $ for all $ k\in\N $, $ \alpha\in(0,1) $. In the mentioned reference, periodic Besov space $ B^s_{\infty,\infty} $ have been used. Recall that $ B^s_{\infty,\infty} $ coincides with the classical H\"older space $ C^{k,\alpha}(\T) $ for $ s=k+\alpha\not\in\N $. This concludes the proof.
	\end{proof}
	\subsubsection*{Proof of Proposition \ref{pro:InvertLinOp}}
	
	We consider Case \eqref{CaseA} and Case \eqref{CaseB} simultaneously, since the proof is the same. Given $ (S,Z,M)\in \mathbb{Z}^{k+1,\alpha}= C^{k+1,\alpha}(\T)\times \R\times \R $ we want to solve for $ (g,b,\mu) \in H^{k+2,\alpha}_0\times \R\times \R$ the equations
	\begin{align}\label{eq:LinearizedSystem}
		\begin{split}
			\Lin g -\mu &= S,
			\\
			\Omega_0^2b-\partial_{x_1}^2U_0(a_0,0)b - W_{0,a_0}[g] &=Z,
			\\
			\pi\hat{g}_0 &= M.
		\end{split}
	\end{align}
	First, we have $ \hat{g}_0 = M/\pi $. For the first equation in \eqref{eq:LinearizedSystem} we decompose $ S $ in its Fourier coefficients $ (\hat{S}_n)_{n\in \Z} $. Then, the first equation in \eqref{eq:LinearizedSystem} becomes
	\begin{align*}
		\sum_{n\geq1} \omega_n \, \hat{g}_n e^{in\varphi}+\sum_{n\leq-1} \omega_n \, \overline{\hat{g}_{|n|}} e^{in\varphi} = \hat{S}_0- \dfrac{2\omega_0M}{\pi} + \mu +  \sum_{n\geq1}\hat{S}_{n}e^{in\varphi} + \sum_{n\leq-1}\overline{\hat{S}_{|n|}}e^{in\varphi}.
	\end{align*}
	Recall that $ \overline{\hat{S}_{-n}}=\hat{S}_{n} $ for $ n\geq0 $ since $ S $ is a real-valued function. We then choose $ \mu=2\omega_0M/\pi-\hat{S}_0  $. Since the multipliers $ \omega_n $ of $ \Lin $ are non-zero by assumption \eqref{eq:NonResonanceAssumption}, we can define $ \Lin^{-1}=OP(\omega_n^{-1}) $. By Lemmas \ref{lem:AsymptotcisCoeff} and \ref{lem:AsymptotcisCoeffGravPot} we can write 
	\begin{align*}
		\omega_n = \dfrac{|n|+b_n}{\kappa}, \quad \kappa^{-1}:= -\phi_0'(1)^2,
	\end{align*}
	with $ \sup_{n\neq0}|b_n||n|^{-\gamma}\leq C $ for any $ \gamma>0 $. Note that by our assumption in Theorem \ref{thm:MainThm} we also have $ \phi_0'(1)\neq0 $. We can hence apply Lemma \ref{lem:PeriodicPseudoDiffReg} yielding $ F\in C_0^{k+2,\alpha}(\T) $ defined by
	\begin{align*}
		F = OP(\omega_n^{-1}) (S-\hat{S}_0).
	\end{align*}
	Note that $ F $ is real-valued with $ \hat{F}_n= \hat{S}_n/\omega_n $ for $ n\geq 1 $.
	
	The function $ F $ is only defined on the torus. We now define the function $ g $ from $ F $ via
	\begin{align}\label{eq:ProofInvertibilityHolomorphicFunct}
		g(z) = \hat{g}_0z + \sum_{n\geq1} \hat{F}_n z^{n+1} = \dfrac{M}{\pi}z + \sum_{n\geq1} \dfrac{\hat{S}_n}{\omega_n} z^{n+1}.
	\end{align}
	We need to show that $ g\in H_0^{k+2,\alpha} $. To this end, define the function $ \tilde{F}:=\frac{1}{2}(I+\mathscr{H})F $, recalling that  $ \mathscr{H} $ denotes the Hilbert transform. The function $ \tilde{F} $ has the Fourier decomposition
	\begin{align*}
		\tilde{F}(\varphi) = \sum_{n\geq1} \hat{F}_n e^{in\varphi}, \quad \norm[C^{k+2,\alpha}(\T)]{\tilde{F}}\leq \norm[C^{k+2,\alpha}(\T)]{F}.
	\end{align*}
	The last inequality follows from the fact that $ \mathscr{H}: C^{k+2,\alpha}(\T)\to C^{k+2,\alpha}(\T) $ is bounded with $ \norm{\mathscr{H}}=1 $. Since $ \tilde{F} $ contains only Fourier modes $ n\geq0 $, there is a unique holomorphic extension in $ C^{k+2,\alpha}(\overline{\D}) $. This extension has the power series expansion 
	\begin{align*}
		\tilde{F}(z)  = \sum_{n\geq1} \hat{F}_n z^{n+1}.
	\end{align*}
	Consequently, the function $ g(z):=\hat{g}_0z + \tilde{F}(z) \in H^{k+2,\alpha}_0 $ satisfies \eqref{eq:ProofInvertibilityHolomorphicFunct} and hence also \eqref{eq:LinearizedSystem}.
	
	Finally, we determine $ b $ in \eqref{eq:LinearizedSystem}. To this end, we use the radial symmetry of the potential $ U_0 $, yielding
	\begin{align*}
		D_Y^2U_0(Y)= D_Y^2[U_0(|Y|)] = U_0'(|Y|) \left( \dfrac{|Y|^2I-Y\otimes Y}{|Y|^3} \right) + U_0''(|Y|) \dfrac{Y\otimes Y}{|Y|^3}.
	\end{align*}
	Hence, we need to solve, putting $ Y=X_0=(a_0,0) $ in the previous formula,
	\begin{align*}
		\left( \Omega_0^2-\dfrac{U_0''(a_0)}{a_0} \right) b = Z + W_{0,a_0}[g].
	\end{align*}
	At this point $ W_{0,a_0}[g] $ is a determined real number. From Lemma \ref{lem:PropertiesUnpertPot} we have $ U_0''(a_0)\leq 0 $. Thus, we can invert the above equation in terms of $ b $.
	
	The above arguments show that $ D_{(h,a,\lambda)}\F(0,a_0,\lambda_0,0) $ is one-to-one and onto. Hence, it is an isomorphism which concludes the proof.

	\section{Proof of Theorem \ref{thm:MainThm} and consequences}\label{sec:main:proof}
	In this last section, we first provide the proof of Theorem \ref{thm:MainThm}. We also include the details towards Corollary \ref{cor:MainResult} which is a direct consequence of the previous main result. 
	
	\subsubsection*{Proof of Theorem \ref{thm:MainThm}}
	Due to Proposition \ref{pro:MainProbFrechetDeriv} the function $ \F $ is continuously differentiable. Under assumption \eqref{eq:NonResonanceAssumption} we can invert the linearized operator $ D_{(h,a,\lambda)}\F(0,a_0,\lambda_0,0) $ by Proposition \ref{pro:InvertLinOp}. Hence, we can apply the implicit function theorem, see Lemma \ref{lem:IFT}. This concludes the proof.
	
	\subsubsection*{Proof of Corollary \ref{cor:MainResult}}
	For the sake of clarity we divide the proof into three steps.
	
	\paragraph{Step 1: Symmetry.} We first prove the symmetry of the domain $ E_h $. To this end, we show that the function $ g(z):=\overline{h(\overline{z}) }\in H^{k+2,\alpha}_0 $ satisfies $ \F(g,a,\lambda,m)=0 $.  Note that $ g $ induces a conformal map $ f_g $ which parameterizes the domain $ R(E_h) $, where $ R(x_1,x_2)=(x_1,-x_2) $. As a consequence of the uniqueness of solutions to \eqref{eq:StreamFunction}, the stream function satisfies $ \psi_g(x) = \psi_h(Rx) $. Furthermore, we have, recalling $ X=(a,0) $,
	\begin{align*}
		U_g(x) &=  U_{R(E_h)}(x) = U_{h}(Rx),
		\\
		U_X(x) &= U_X(Rx). 
	\end{align*}
	Since $ (h,a,\lambda,m) $ is a solution, we obtain from \eqref{eq:FreeBoundaryEq}, which is equivalent to the first component of $ \F $, and application of $ x\mapsto Rx $
	\begin{align*}
		\dfrac{1}{2}|\nabla^\perp\psi_g(x)|^2 - \dfrac{\Omega_0^2}{2}|x|^2 + U_g(x) +mU_X(x) = \lambda \quad x\in \partial E_g,
	\end{align*}
	The other components of $ \F(g,a,\lambda,m)=0 $ follow in the same manner. By the uniqueness statement of the implicit function theorem we have $ f_h(z)=f_g(z)=\overline{f_h(\overline{z})} $, i.e. the domain $ E_h $ is symmetric. 
	
	\paragraph{Step 2: Solution.} The symmetry of the domain $ E_h $ implies
	\begin{align*}
		\partial_{x_2}U_h(X) = 0 &= \Omega_0^2X_2.
	\end{align*}
	We can now define the non-hydrostatic pressure $ P $ as in \eqref{eq:ReconstructionPressure} and observe that all equations but the last one in the system \eqref{eq:IncompEulerPoisson} are satisfied for $ v=\nabla^\perp \psi_h $, $ X=(a,0) $ and $ P $.
	
	\paragraph{Step 3: Center of mass.}	We now show that the last equation in \eqref{eq:IncompEulerPoisson} is a consequence of the other equations in \eqref{eq:IncompEulerPoisson}. More precisely, they imply that the center of mass is zero
	\begin{align*}
		X_{c} := \dfrac{1}{\pi+m}\left( \int_{E_h}x\, dx +mX\right)  = 0.
	\end{align*}
	Combining the first equation and the fifth equation in \eqref{eq:IncompEulerPoisson} gives
	\begin{align*}
		(\pi+m)\, \Omega_0^2 \, X_{c} = \int_{E_h} \left( (v\cdot \nabla)v  + 2\Omega_0  Jv + \nabla P \right) \, dx + m\nabla U_h(X).
	\end{align*}
	Since $ (v\cdot \nabla)v = \divergence (v\otimes v) $ and $ v\cdot n_h=0 $ on $ \partial E_h $ the first term is zero. Furthermore, due to $ v=\nabla^\perp \psi_h= J\nabla\psi_h $ and $ \psi_h=0 $ on $ \partial E_h $ we have
	\begin{align*}
		\int_{E_h} Jv \, dx = -\int_{E_h} \nabla \psi_h \, dx = -\int_{\partial E_h} \psi_h \, n_h \, dS = 0.
	\end{align*}
	We have for the non-hydrostatic pressure 
	\begin{align*}
		\int_{E_h}\nabla P \, dx = \int_{E_h} P \, n_h \, dS = \int_{\partial E_h} \left( U_h+mU_X \right)  n_h \, dS,
	\end{align*}
	where we used the fourth equation in \eqref{eq:IncompEulerPoisson}. Furthermore, we have in Case \eqref{CaseA}
	\begin{align*}
		m\nabla U_h(X) = -m \int_{E_h} \nabla_y \left[ \dfrac{1}{|X-y|^\nu} \right] \, dy = -m \int_{\partial E_h} \dfrac{n_h}{|X-y|^\nu}\, dS(y) = -m\int_{\partial E_h} U_X \, n_h \, dS.
	\end{align*} 
	In Case \eqref{CaseB} we get a corresponding equality. This yields 
	\begin{align*}
		(\pi+m)\, \Omega_0^2 \, X_{c} =\int_{\partial E_h} U_h\,  n_h \, dS.
	\end{align*}
	By symmetry of the interaction potential we obtain in Case \eqref{CaseA}
	\begin{align*}
		\int_{\partial E_h} U_h\,  n_h \, dS = \int_{E_h} \nabla U_h(x) \, dx = \nu\int_{E_h}\int_{E_h} \dfrac{x-y}{|x-y|^{\nu+2}} \, dxdy = 0.
	\end{align*}
	However, this argument holds only for $ \nu<1 $ due to the singularity. For $ \nu=1 $ we use an approximation. The same conclusion holds in Case \eqref{CaseB}. This implies $ X_c=0 $, since $ \Omega_0\neq 0 $, which concludes the proof.
	
	\bibliographystyle{habbrv}
	\bibliography{References}

\begin{thebibliography}{10}

\bibitem{Abramowitz1965Handbook}
M.~Abramowitz and I.~Stegun.
\newblock {\em Handbook of Mathematical Functions: With Formulas, Graphs, and
  Mathematical Tables}.
\newblock Applied mathematics series. Dover Publications, 1965.

\bibitem{AltCaffarellFriedman1982AsymJetFlows}
H.~W. Alt, L.~A. Caffarelli, and A.~Friedman.
\newblock Asymmetric jet flows.
\newblock {\em Comm. Pure Appl. Math.}, 35(1):29--68, 1982.

\bibitem{AltCaffarellFriedman1982JetGravity}
H.~W. Alt, L.~A. Caffarelli, and A.~Friedman.
\newblock Jet flows with gravity.
\newblock {\em J. Reine Angew. Math.}, 331:58--103, 1982.

\bibitem{AltCaffarellFriedman1982AxiallySymmJetFlows}
H.~W. Alt, L.~A. Caffarelli, and A.~Friedman.
\newblock Axially symmetric jet flows.
\newblock {\em Arch. Rational Mech. Anal.}, 81(2):97--149, 1983.

\bibitem{AuchmutyBeals1971ModelsRotatingStars}
J.~F.~G. {Auchmuty} and R.~{Beals}.
\newblock Models of rotating stars.
\newblock {\em Astrophysical Journal}, 165:79--82, 1971.

\bibitem{burbea1982motions}
J.~Burbea.
\newblock Motions of vortex patches.
\newblock {\em Letters in Mathematical Physics}, 6:1--16, 1982.

\bibitem{Burckel1979CompAna}
R.~B. Burckel.
\newblock {\em An introduction to classical complex analysis. {V}ol. 1},
  volume~82 of {\em Pure and Applied Mathematics}.
\newblock Academic Press, New York-London, 1979.

\bibitem{Cardona2017HolderPeriodicPseudoDiff}
D.~Cardona.
\newblock H\"{o}lder-{B}esov boundedness for periodic pseudo-differential
  operators.
\newblock {\em J. Pseudo-Differ. Oper. Appl.}, 8(1):13--34, 2017.

\bibitem{Chandrasekhar1969}
S.~Chandrasekhar.
\newblock {\em Ellipsoidal Figures of equilibrium}.
\newblock Yale University Press, 1969.

\bibitem{ChanilloLi1994OnDiametersOfUniformly}
S.~Chanillo and Y.~Y. Li.
\newblock On diameters of uniformly rotating stars.
\newblock {\em Comm. Math. Phys.}, 166(2):417--430, 1994.

\bibitem{FaaDiBruno}
G.~M. Constantine and T.~H. Savits.
\newblock A multivariate {F}a\`a di {B}runo formula with applications.
\newblock {\em Trans. Amer. Math. Soc.}, 348(2):503--520, 1996.

\bibitem{Deimling}
K.~Deimling.
\newblock {\em Nonlinear Functional Analysis}.
\newblock Springer, 2 edition, 1985.

\bibitem{Dolbeault2008LoclaizedMinimizersFlatRotating}
J.~Dolbeault and J.~Fern\'{a}ndez.
\newblock Localized minimizers of flat rotating gravitational systems.
\newblock {\em Ann. Inst. H. Poincar\'{e} C Anal. Non Lin\'{e}aire},
  25(6):1043--1071, 2008.

\bibitem{Giaquinta2013Regularity}
M.~Giaquinta and L.~Martinazzi.
\newblock {\em An Introduction to the Regularity Theory for Elliptic Systems,
  Harmonic Maps and Minimal Graphs}.
\newblock Publications of the Scuola Normale Superiore. Scuola Normale
  Superiore, 2013.

\bibitem{GilbargTrudinger2001}
D.~Gilbarg and N.~S. Trudinger.
\newblock {\em Elliptic Partial Differential Equations of Second Order}.
\newblock Springer Berlin Heidelberg, 2001.

\bibitem{Grad1967Plasma}
H.~Grad.
\newblock Toroidal containment of a plasma.
\newblock {\em The Physics of Fluids}, 10(1):137--154, 1967.

\bibitem{hassainia2020global}
Z.~Hassainia, N.~Masmoudi, and M.~H. Wheeler.
\newblock Global bifurcation of rotating vortex patches.
\newblock {\em Communications on Pure and Applied Mathematics},
  73(9):1933--1980, 2020.

\bibitem{Heilig1994Lichtenstein}
U.~Heilig.
\newblock On {L}ichtenstein's analysis of rotating {N}ewtonian stars.
\newblock {\em Ann. Inst. H. Poincar\'{e} Phys. Th\'{e}or.}, 60(4):457--487,
  1994.

\bibitem{hmidi2013boundary}
T.~Hmidi, J.~Mateu, and J.~Verdera.
\newblock Boundary regularity of rotating vortex patches.
\newblock {\em Archive for Rational Mechanics and Analysis}, 209:171--208,
  2013.

\bibitem{IonescuPusateri}
A.~Ionescu and F.~Pusateri.
\newblock Global solutions for the gravity water waves system in 2d.
\newblock {\em Invent. math.}, 199:653--804, 2015.

\bibitem{JangMakino2017SlowlyRotatingAxissymmetricSol}
J.~Jang and T.~Makino.
\newblock On slowly rotating axisymmetric solutions of the {E}uler-{P}oisson
  equations.
\newblock {\em Arch. Ration. Mech. Anal.}, 225(2):873--900, 2017.

\bibitem{JangMakino2019RotatingAxisymmetricSol}
J.~Jang and T.~Makino.
\newblock On rotating axisymmetric solutions of the {E}uler-{P}oisson
  equations.
\newblock {\em J. Differential Equations}, 266(7):3942--3972, 2019.

\bibitem{JangSeok2022RotatingBinaryStarsGalaxies}
J.~Jang and J.~Seok.
\newblock On uniformly rotating binary stars and galaxies.
\newblock {\em Arch. Ration. Mech. Anal.}, 244(2):443--499, 2022.

\bibitem{Li1991OnUniformlyRotatingStars}
Y.~Y. Li.
\newblock On uniformly rotating stars.
\newblock {\em Arch. Rational Mech. Anal.}, 115(4):367--393, 1991.

\bibitem{Lichtenstein1918UntersuchungenI}
L.~Lichtenstein.
\newblock Untersuchungen \"{u}ber die {G}leichgewichtsfiguren rotierender
  {F}l\"{u}ssigkeiten, deren {T}eilchen einander nach dem {N}ewtonschen
  {G}esetze anziehen. {E}rste {A}bhandlung. {H}omogene {F}l\"{u}ssigkeiten.
  {A}llgemeine {E}xistenzs\"{a}tze.
\newblock {\em Math. Z.}, 1(2-3):229--284, 1918.

\bibitem{Lichtenstein1933UntersuchungenIII}
L.~Lichtenstein.
\newblock Untersuchungen \"{u}ber die {G}leichgewichtsfiguren rotierender
  {F}l\"{u}ssigkeiten, deren {T}eilchen einander nach dem {N}ewtonschen
  {G}esetze anziehen. {D}ritte {A}bhandlung. {N}ichthomogene
  {F}l\"{u}ssigkeiten. {F}igur der {E}rde.
\newblock {\em Math. Z.}, 36(1):481--562, 1933.

\bibitem{LuoSmoller2004RotatingFluids}
T.~Luo and J.~Smoller.
\newblock Rotating fluids with self-gravitation in bounded domains.
\newblock {\em Arch. Ration. Mech. Anal.}, 173(3):345--377, 2004.

\bibitem{LuoSmoller2009ExistenceNonlinearStabilityRotatingStar}
T.~Luo and J.~Smoller.
\newblock Existence and non-linear stability of rotating star solutions of the
  compressible {E}uler-{P}oisson equations.
\newblock {\em Arch. Ration. Mech. Anal.}, 191(3):447--496, 2009.

\bibitem{Rein1999FlatSteadyStates}
G.~Rein.
\newblock Flat steady states in stellar dynamics --existence and stability.
\newblock {\em Communications in Mathematical Physics}, 205(1):229--247, 1999.

\bibitem{Rein2007}
G.~Rein.
\newblock Chapter 5 - {C}ollisionless kinetic equations from astrophysics: The
  {V}lasov-{P}oisson system.
\newblock In C.~Dafermos and E.~Feireisl, editors, {\em Handbook of
  Differential Equations: Evolutionary Equations}, volume~3, pages 383--476.
  North-Holland, 2007.

\bibitem{Rudin1987ComplexAna}
W.~Rudin.
\newblock {\em Real and Complex Analysis}.
\newblock Mathematics series. McGraw-Hill, 1987.

\bibitem{Safra}
V.~D. Shafranov.
\newblock Equilibrium of a plasma toroid in a magnetic field.
\newblock {\em Soviet Physics. JETP}, 10:775--779, 1960.

\bibitem{StraussWu2017SteadyStatesRotatingStarsGalaxies}
W.~A. Strauss and Y.~Wu.
\newblock Steady states of rotating stars and galaxies.
\newblock {\em SIAM J. Math. Anal.}, 49(6):4865--4914, 2017.

\bibitem{StraussWu2019RapidlyRotatingStars}
W.~A. Strauss and Y.~Wu.
\newblock Rapidly rotating stars.
\newblock {\em Comm. Math. Phys.}, 368(2):701--721, 2019.

\bibitem{Sijue1}
S.~Wu.
\newblock Well-posedness in {S}obolev spaces of the full water wave problem in
  {$2$}-{D}.
\newblock {\em Invent. Math.}, 130(1):39--72, 1997.

\bibitem{Sijue2}
S.~Wu.
\newblock Global wellposedness of the 3-{D} full water wave problem.
\newblock {\em Invent. Math.}, 184(1):125--220, 2011.

\end{thebibliography}
\end{document}